\documentclass[reqno,11pt]{amsart}
\usepackage{amssymb,amsmath,amsthm,amstext,amsfonts}
\usepackage[dvips]{graphicx}
\usepackage{a4wide}
\usepackage{psfrag}
\usepackage{color}

\theoremstyle{plain}
\newtheorem{maintheorem}{Theorem}

\newtheorem{maincorollary}{Corollary}

\newtheorem{theorem}{Theorem }[section]
\newtheorem{proposition}[theorem]{Proposition}
\newtheorem{lemma}[theorem]{Lemma}

\theoremstyle{definition} \theoremstyle{remark}
\newtheorem{remark}[theorem]{Remark}
\newtheorem{example}[theorem]{Example}
\newtheorem{definition}[theorem]{Definition}

\newcommand{\field}[1]{\mathbb{#1}}

\newcommand{\NN}{\field{N}}

\newcommand{\topp}{\operatorname{top}}
\newcommand{\Ptop}{P_{\topp}}

\newcommand{\al} {\alpha}

\newcommand{\de} {\delta}       

\newcommand{\vep}{\varepsilon}
\renewcommand{\epsilon}{\varepsilon}

      \newcommand{\La}{\Lambda}

\newcommand{\si} {\sigma}

\newcommand{\cC}{\mathcal{C}}

\newcommand{\cP}{\mathcal{P}}

\newcommand{\cM}{\mathcal{M}}

\newcommand{\cQ}{\mathcal{Q}}

\newcommand{\Leb}{\mbox{Leb}}

\newcommand{\R}{{\mathbb R}}

\begin{document}

\title[Gluing orbit property, hyperbolicity and large deviations for flows]
{The gluing orbit property, uniform hyperbolicity and large  deviations principles for semiflows}

\author{Thiago Bomfim and Paulo Varandas}

\address{Departamento de Matem\'atica,
  Universidade Federal da Bahia\\
  Av. Ademar de Barros s/n, 40170-110 Salvador, Brazil\\
  }
\email{paulo.varandas@ufba.br}
\urladdr{www.pgmat.ufba.br/varandas}
\email{tbnunes@ufba.br}
\urladdr{https://sites.google.com/site/homepageofthiagobomfim/}

\date{\today}

\begin{abstract}
In this article we introduce a gluing orbit property, weaker than specification, for both maps and flows. We prove that flows
with the $C^1$-robust gluing orbit property are uniformly hyperbolic and that every uniformly hyperbolic flow satisfies the
gluing orbit property. We also prove a level-1 large deviations principle and a level-2 large deviations lower bound for
semiflows with the gluing orbit property. As a consequence we establish a level-1 large deviations
principle for hyperbolic flows and every continuous observable, and also a level-2 large deviations lower bound.
%An exponential large deviation upper bound for continuous observables on suspension semiflows over a non-uniformly expanding base transformation with non-flat singularities or criticalities (where the roof function
%behaves like the logarithm of the distance to the singular/critical set of the base map) for the SRB measure
%of non-uniformly hyperbolic suspension semiflows was obtained by Ara\'ujo~\cite{Ar07}. Later Ara\'ujo and Bufetov~\cite{AB11} extended the work of LS.Young~\cite{You90} and proved a large deviation principle
%for suspension flows over symbolic dynamical systems.
%Large deviations lower are much harder to obtain since many times one requires some
%specification property. While the specification property is unlikely to hold for most flows (c.f. Sumi,
%Varandas, Yamamoto~\cite{SVY15}) our purpose here is to introduce a gluying orbit property for
%flows which is enough to guarantee the large deviations lower bounds with respect to weak Gibbs measures.
%As an application we obtain large deviation principles for hyperbolic flows and continuous observables, extending the
%large deviations for hyperbolic flows and H\"older observables obtained by Kifer~\cite{Kifer} and
%Waddington~\cite{Wad96}.
Finally, since many non-uniformly hyperbolic flows can be modeled as suspension flows we also provide
criteria for such flows to satisfy uniform and non-uniform versions of the gluing orbit property.
\end{abstract}

   % 	Stability theory
\maketitle

%\tableofcontents

%%%%%%%%%%%%%%%%%%%%%%%%%%%%%%%%%%%%%%%%%%%
\section{Introduction}

After the notion of uniform hyperbolicity has been introduced in the seventies by Smale~\cite{Smale1967},
the study of the thermodynamical formalism for uniformly hyperbolic maps and flows has drawn the attention
of many researchers. The construction of physical, Sinai-Ruelle-Bowen and equilibrium measures and the study
of their statistical properties are some well studied topics. Among the statistical properties, the rates of decay of correlations
and large deviations turned out to be much more difficult problem in the time-continuous setting rather than for discrete
time dynamics.
In fact, while for uniformly hyperbolic diffeomorphisms every H\"older continuous potential admits a unique equilibrium state, which is a Gibbs measure and has exponential decay of correlations (see ~\cite{Bo75, Ru76, Si68})
the counterpart of these mixing results for hyperbolic flows was soon proved to be false. Examples of flows that are uniformly hyperbolic but with arbitrarily slow mixing rates were given by Ruelle~\cite{ruelle1983} and later studied by Pollicott~\cite{Pol}.
For surveys on mixing rates for hyperbolic flows we refer the reader towards the introductions of~\cite{liverani2004, FMT}.

In the nineties, Young, Kifer and Newhouse~\cite{You90, Kifer, KiferN} addressed the question of the velocity
of convergence of ergodic averages establishing a connection between the theory of large deviations in probability to the realm of dynamical systems, a topic that has given much description of the chaotic
features of dynamical systems. L.-S. Young's thermodynamical approach to provide large deviations principles for
Gibbs measures and all continuous observables usually requires the uniqueness of equilibrium states and some form
of specification, which is common among hyperbolic diffeomorphisms. Indeed, every diffeomorphism restricted to
a topologically mixing hyperbolic set satisfies the specification property (see e.g.~\cite{Katok}).
Other approaches to large deviations whenever the pressure function is differentiable, as the one used by Kifer~\cite{Kifer},
lead to stronger results although often require observables to be at least H\"older continuous.

For uniformly hyperbolic flows a unified method for large deviations using the thermodynamical approach of \cite{You90}
and the specification property drops dramatically since uniformly hyperbolic flows may be even not topologically mixing.
Nevertheless, Kifer~\cite{Kifer} and Waddington~\cite{Wad96}, among other limit theorems, established a
large deviations principle for hyperbolic flows and regular observables (at least H\"older continuous).
%These are modeled by suspension flows over
%subshifts of finite time, that model Axiom A flows (see e.g.~\cite{BR75}).
While good spectral properties of transfer operators imply in other strong consequences, its extension for
a broad non-uniformly hyperbolic context usually requires a ``case by case" study. To push further the analysis and to be
able to consider more general continuous observables, it is natural to introduce other tool that could replace specification
as a mechanism to prove large deviations principles.
In fact, the recent revived interest for specification properties and large deviations in the last decade
shows that the original idea of specification, which corresponds to a strong shadowing of
pieces of orbits, introduced by Bowen~\cite{Bo71}, is far from generating an old fashioned mechanism
to study the topological and ergodic features of the dynamical system. %Indeed, many of
%its extensions to non-uniformly hyperbolic transformations still provide important characterization of the
%ergodic features of the dynamical system.
While the strong specification property fails to extend beyond
uniformly hyperbolic diffeomorphisms and flows (c.f.\cite{SVY14,SVY15})) many other non-uniform notions
have been introduced to reflect non-uniform hyperbolicity (c.f. \cite{PS05,OT,Va12}).
In particular one expects the gluing orbit property to be an useful tool to replace the specification
property e.g. in the study of multifractal formalism for non-uniformly hyperbolic flows. Just as an illustration
the gluing orbit property can be proved to hold for suspension flows over the Manneville-Pomeau. We refer the reader to
Section~\ref{sec:examples} for some examples. A similar notion of gluing for $C^1$-diffeomorphisms was introduced in \cite{ST},
referred as transitive specification property, where the authors prove that this is equivalent to uniform hyperbolicity.

In this paper we shall address on the ergodic theory of semiflows with the gluing orbit property and also provide
a characterization of $C^1$-smooth flows for which this property holds robustly.
One first goal here is to prove large deviations estimates for %weak Gibbs measures of
semiflows with the gluing orbit property.
We prove a level-1 large deviations principle for any \emph{continuous} observable and also prove a level-2 large deviations lower
bound for semiflows with the gluing orbit property. In both cases, the estimates and the the rate function are expressed in
terms of the thermodynamical  quantities and probability measures that invariant either by the time-one map or by the flow
(c.f. Theorems~\ref{thm:LB} and ~\ref{thm:LB2}).
Since Axiom A flows are semi conjugate to the suspension
flows over subshifts of  finite type, and these
satisty the above mentioned property (as a consequence of Theorem~\ref{thm:gluingAA} in Section~\ref{sec:statements}),
then a level-1 large deviation principle holds for every transitive hyperbolic flow and every continuous observable.
Even in the hyperbolic case our results provide a simpler proof of the level-1 large deviations considered in~\cite{Wad96}, applies
to a wider class of observables and yields a level-2 large deviations lower bound.
%which has not been obtained was not proven before.
%These include large deviations principles for SRB and other equilibrium measures associated to
%suspension semiflows by countable shifts (see e.g. \cite{Pinheiro}), for which Ara\'ujo and Bufetov~\cite{AB11}
%also obtained large deviations principles.
Let us mention that important level-1 large deviations estimates for non-uniformly hyperbolic flows
were obtained e.g. by Melbourne and Nicol~\cite{Mel,MN08}, Ara\'ujo~\cite{Ar07} and Ara\'ujo and Bufetov~\cite{AB11},
where the observables considered are required to have larger regularity than continuity. Most of these
results only consider large deviations upper
bounds.
%The later presents also lower bounds for compactly supported observables.
%Some large deviations upper bounds have been obtained also by Climenhaga and Thompson~\cite{CT15} for maps
%with non-uniform structures.
%
A second goal here is, in view of the previous discussion, to ask whether if, under some additional conditions, the
specification and gluing orbit properties do coincide. Such extra conditions could be from a topological nature (e.g. topological mixing)
or on the smooth structure (e.g. the conditions to hold robustly within a $C^1$ neighborhood of the original flow).
Motivated by the results of Sakai, Sumi, Yamamoto~\cite{SSY} and their extension for flows by Arbieto, Senos, Sodero~\cite{AST}
we prove that $C^1$-robustly, the gluing orbit and the specification properties are equivalent to the topological
mixing and uniform hyperbolicity of the flow (see Theorem~\ref{thm:robust} and Corollary~\ref{cor:robust}).
Finally, motivated by the fact that many flows can be modeled by suspension semiflows, we prove some criteria
for suspension semiflows to satisfy the gluing orbit property.

This article is organized as follows.
Definitions and the statement of our main results are given in Section~\ref{sec:statements}. In Section~\ref{sec:examples}
we give some examples to which our results apply while in Section~\ref{open} we shall make further comments
and discuss some open questions.
Section~\ref{proofs} is devoted to the proof of the main results concerning the gluing orbit property
and its relation with hyperbolicity. In Section~\ref{sec:LDP} we use the gluing orbit property to provide
large deviations upper and lower bounds and establish large deviation principles for flows. In section~\ref{sec:criteria} we provide criteria for such flows satisfy uniform and non-uniform versios of the gluing orbit property. Finally, we include an Appendix where
we discuss , for suspension flows, the relation between a tempered variation condition for observables on the manifold
and the same condition for the reduced observable on the base dynamics.
%In particular
%we give an alternative proof to the large deviations principle for Axiom A flows given in \cite{Wad96}.

%%%%%%%%%%%%%%%%%%%%%%%%%%%%%%%%%%%%%%%%%%%%%%%%%%%%%%%%%%%%%%%%%%%
\section{Preliminaries and Statement of the main results}
\label{sec:statements}

%Throughout, let $M$ be a compact Riemannian manifold, let
%$d$ denote the induced Riemannian distance in $M$,
%$\|\cdot\|$ the Riemannian norm and $\Leb$ the induced
%normalized Riemannian volume form.

%%%%%%%%%%%%%%%%%%%%%%%%%%%%%%%%%%%%%%%%%%%%%%
\subsection{Preliminaries}

In this section we shall recall some notions that will be necessary for the understanding of our main results
and introduce two notions of a gluing orbit property. The reader may decide to skip this section in a first
reading and to return to it whenever its makes necessary for the understanding of the article.
%

%%%%%%%%%%%%%%%%%%%%%%%%%%%%%%%%%%%%%%%%%%%%%%
\subsubsection{Hyperbolic, sectional-hyperbolic and suspension flows}\label{sec:suspension}

In this subsection we recall some preliminaries on suspension semiflows, uniform and sectional hyperbolicity for flows.

%%%%%%%%%%%%%%%%%%%%%%%%%%%%%%%%%%%%%%%%%%%%%%
\subsubsection*{Suspension semiflows}\label{subsec:suspension}

Assume that $M$ be a measurable space and $f$ be a measurable map on $M$. Given an
$f$-invariant probability measure $\mu$ and a measurable \emph{roof function}
$\rho \colon M \to [0, +\infty)$ we define the \emph{suspension semiflow}  $(X_t)_{t\ge 0}$ over $f$
by  $X_t(x,s)=(x,s+t)$, acting on the quotient space
$$
M_\rho = \{ (x,t) \in M \times \mathbb R_0^+ :  0 \le t \le \rho(x)\} / \sim
$$
where $\sim$ is the equivalence relation given by $(x,\rho(x))\sim (f(x),0)$.  In these coordinates $(X_t)_t$ coincides
with the flow consisting in the displacement along the ``vertical" direction. If $f$ is invertible and $\rho \in L^1(\mu)$
it is not difficult to check that $(X_t)_t$ defines a flow and it preserves the probability measure
$
\overline\mu=(\mu \times \Leb) / \int \rho \, d\mu,
$
where $\Leb$ denotes the Lebesgue measure on the real line.
Furthermore, observe that $\overline \mu$ is uniquely defined by the previous expression provided
the roof function $\rho$ is bounded away from zero. Given $\psi: M_{\rho} \rightarrow \mathbb{R}$ we associate the observable $\overline{\psi} : M \rightarrow \mathbb{R}$ defined as $\overline{\psi}(x) = \int_{0}^{\rho(x)}\psi(x , t)dt$.
We endow the space $M_\rho$ with the Bowen-Walters distance (we refer the reader to the beginning of Section~\ref{sec:criteria} for the precise definition).

%%%%%%%%%%%%%%%%%%%%%%%%%%%%%%%%%%%%%%%%%%%%%%
\subsubsection*{Hyperbolic and sectional-hyperbolic flows}

Let $M$ be a closed Riemannian manifold,
$d$ denote the induced Riemannian distance in $M$,
$\|\cdot\|$ the Riemannian norm. %and $\Leb$ the induced normalized Riemannian volume form.
Let  $(X_t)_t$ be a smooth flow on $M$ and $\Lambda\subseteq M$ be a compact and $(X_t)_t$-invariant set.
We say that the flow $(X_t)_t$ to $\La$ is  \emph{uniformly hyperbolic} on $\Lambda$ (or simply that $\Lambda$ is a uniformly hyperbolic set) if there exists a $DX_t$-invariant
and continuous splitting $T_\La N= E^-\oplus X \oplus E^+$ and
constants $C>0$ and $0<\theta_1<1$ such that
$$
\|DX_t \mid E^-\| \leq C \theta_1^t
    \quad\text{and}\quad
\|(DX_t)^{-1} \mid E^+\| \leq C \theta_1^t,
    \; \forall t \geq 0
$$
for every $x \in M$.  A flow $(X_t)_t$ is said to be (i) \emph{Anosov} if the whole manifold $M$ is a hyperbolic set; and
(ii) \emph{Axiom A} if its non-wandering set $\Omega$ is a hyperbolic set with a dense subset of periodic orbits.
Uniformly hyperbolic flows have been well
studied since the 1970's and, in particular, their geometric
structure is very well understood. It follows from the work
of Bowen, Sinai and Ruelle~\cite{BR75,Bo75,Si68} that hyperbolic flows admit finite
Markov partitions and that are semi-conjugated to suspension
flows over subshifts of finite type.

We say that a $X_t$-invariant compact set $\Lambda$ is \emph{sectional-hyperbolic}
if every singularity in $\Lambda$ is hyperbolic and there exist a continuous non-trivial invariant splitting $T_\Lambda M=E^s\oplus E^c$ over $\Lambda$ and constants $C > 0$ and $\lambda \in (0,1)$ such that
for every $x\in \Lambda$ and $t\geq0$
\begin{itemize}
\item[(i)] $\|DX_t\mid E^s_x\| \; \|DX_{-t}\mid E^c_{X_t(x)} \| < C \lambda^t $;
\item[(ii)] $\|DX_t(x) \mid E^s_x \| \leq C \lambda^t$;
\item[(iii)]  $|\det(DX_t(x)\mid_{L_x})|>C\lambda^t$ for every plane $L_x\subset F_x$.
\end{itemize}
We say that $p$ is a hyperbolic critical element if $p$ is either a hyperbolic singularity or
a hyperbolic periodic orbit.
\subsubsection{Specification and gluing orbit properties}\label{d.strong.specification}
Let us first recall some specification properties in the discrete time setting.
%
%\begin{definition}
We say that a continuous map $f: X \to X$ on a compact metric space $X$
satisfies the \emph{specification property} if
for any $\vep>0$ there exists an integer $N=N(\vep)\geq 1$ such that
the following holds: for every $k\geq 1$, any points $x_1,\dots,
x_k$, and any sequence of positive integers $n_1, \dots, n_k$ and
$p_1, \dots, p_k$ with $p_i \geq N(\vep)$
there exists a point $x$ in $M$ such that
$
\begin{array}{cc}
d(f^j(x),f^j(x_1) ) \leq \vep, %&\forall \,0\leq j \leq n_1
\end{array}
$
for every $0\leq j \leq n_1$ and
$$
\begin{array}{cc}
d\big(f^{j+n_1+p_1+\dots +n_{i-1}+p_{i-1}}(x) \;,\; f^j(x_i)\big)
        \leq \vep &
\end{array}
$$
for every $2\leq i\leq k$ and $0\leq j\leq n_i$.
%\end{definition}
Topologically mixing subshifts of finite type are among the class of transformations that
satisfy the specification property. %a strong property of shadowing pieces of orbits.
Other measure theoretical non-uniform versions of the specification property have been
introduced (see e.g.~\cite{PS05,OT,Va12}). Following, \cite{Va12} we say
that $(f,\mu)$ satisfies the \emph{non-uniform specification property}
if there exists $\de>0$ such that for $\mu$-almost every $x$ and
every $0<\vep<\de$ there exists an integer $p(x,n,\vep)\geq 1$
satisfying
$$
\lim_{\vep\to 0}\limsup_{n \to \infty} \frac{1}{n} p(x,n,\vep)=0
$$
and so that the following holds: given points $x_1, \dots, x_k$ in a full $\mu$-measure set
and positive integers $n_1, \dots, n_k$, if $p_i \geq p(x_i,n_i,\vep)$
then there exists $z$ that $\vep$-shadows the orbits of each $x_i$
during $n_i$ iterates with a time lag of $p(x_i,n_i,\vep)$ in
between $f^{n_i}(x_i)$ and $x_{i+1}$, that is,
$$
z \in B(x_1,n_1,\vep)
    \quad\text{and}\quad
f^{n_1+p_1+\dots +n_{i-1}+p_{i-1}}(z) \in B(x_i,n_i,\vep)
$$
for every $2\leq i\leq k.$ Here
$
B(x,n,\vep):=\{ y\in X \colon d(f^j(x), f^j(y)) <\vep, \; \forall j=0\dots n-1 \}
$
is the usual Bowen ball of length $n$ and size $\vep$ around $x$.

In the context of flows, we say that the flow $(X_t)_{t\in \mathbb R}$ has the  \emph{specification property}
on $\Lambda\subset M$
if for any $\epsilon> 0$ there exists a $T= T(\epsilon)>0$ such that
the following property holds: given any finite colection of intervals $I_i=[a_i, b_i]$ ($i=1\dots m$)
of the real line satisfying $a_{i+1} - b_i\ge T(\epsilon)$ for every $i$ and every map
$P:\bigcup_{I_i \in \tau}I_i \to \Lambda$ such that $X_{t_2}(P(t_1))=X_{t_1}(P(t_2))$ for any $t_1,t_2\in I_i$
there exists  $x\in\Lambda$ so that $d(X_t(x), P(t)) < \epsilon$ for all $t\in \bigcup_i I_i$.

Since the later properties of specification imply on topologically mixing and we need to consider more general
transitive dynamics we were led to introduce the following notions.

\begin{definition}\label{def:gluing1} (Uniform gluing for homeomorphisms)
We say a continuous map $f: M \to M$ on a compact metric space $M$
satisfies the \emph{gluing orbit property} if for any $\vep>0$ there exists an integer $N=N(\vep)\geq 1$
so that for any points $x_1, x_2, \dots, x_k \in M$ and any positive integers $n_1, \dots, n_k$
there are $p_1, \dots, p_k  \le N(\vep)$ and a point $x$ in $M$ so that
$
d(f^j(x),f^j(x_1) ) \leq \vep
$
for every $0\le j \le n_1$ and
$$
\begin{array}{cc}
d\big(f^{j+n_1+p_1+\dots +n_{i-1}+p_{i-1}}(x) \;,\; f^j(x_i)\big)
        \leq \vep &
\end{array}
$$
for every $2\leq i\leq k$ and $0\leq j\leq n_i$.
\end{definition}

%%%%%%%%%%%%%%%%%%%%%%%%%%%%%%%%%%%%%%%%%%%%%%%
%\subsection{Flows with the gluying orbit property}
As mentioned above Axiom A flows are semi-conjugate to suspension flows over subshifts
of finite type. Consequently, many important ergodic properties including the thermodynamical
formalism of hyperbolic flows can be established using the reduction to the base dynamics
(see e.g. \cite{BR75}). Bowen~\cite{Bow72} characterized the Axiom A
flows that exhibit the specification property, crucial to deduce lower bound estimates for
large deviations using a similar thermodynamical approach to~\cite{You90}, and in particular suspension flows
with a roof function cohomologous to a constant never satisfy the specification property. In other words,
any Axiom A flow whose stable and unstable manifolds are jointly integrable is not topologically mixing,
hence it does not satisfy the specification property (we refer the reader to \cite{Bow72} for more details).
Thus we shall consider also a gluing orbit property for semiflows as follows.

\begin{definition}\label{def:gluing2} (Gluing orbit property for semiflows)
Let $(X_t)_{t\ge 0}$ be a semiflow (not necessarily suspension flow) on a separable metric space $M$.
We say that $(X_t)_t$ has the \emph{gluing orbit property} if for any $\vep>0$ there exists $T(\vep)>0$ so that
for any points $x_1, x_2, \dots, x_k \in M$ and times $t_1, \dots, t_k \ge 0$
there exists $p_1, \dots, p_k  \le T(\vep)$ and $y\in M$ so that
$$
d( X_t(y)), X_t(x_1) ) <\vep
    \quad \forall t \in [0, t_1]
$$
and, if $\underline{x}_i= X_{\sum_{j=0}^{i-1} p_j+ t_j}(y) \in M$ then
$$
d( X_{t}  (\underline{x}_i) , X_t(x_i) ) <\vep
    \quad \forall t \in [0, t_i]
$$
for every $2 \le i \le k$. We say the flow $(X_t)_{t\in\mathbb R}$ satisfies the gluing orbit property
if the semiflows $(X_t)_{t\ge 0}$ and $(X_{-t})_{t\ge 0}$ satisfy this property. We let
$
B(x,t,\vep):=\{ y\in X \colon d(X_s(x), X_s(y)) <\vep, \; \forall s\in [0,t] \}
$
denote the Bowen ball of size $\vep$ and length $t$ around $x$.
\end{definition}

%comparison weak mixing!!

The previous definition roughly means that one can shadow the prescribed pieces of orbits
by a real orbit and that the time length needed from one piece to the following can be bounded
by some time $T(\vep)$ depending only on the proximity $\vep$.
%The shadowing is only required
%to hold at time $T(\vep)$ and, consequently, this property is weaker than
%the specification property.
%
Although the gluing orbit property has the flavor of specification, it is not likely that strong
consequences of the later property can be derived under the first much weaker condition.
A first evidence is that under the gluing orbit property the dynamical is not necessarily
topologically mixing. Finally, notice that the gluing orbit property is clearly a topological invariant.

%%%%%%%%%%%%%%%%%%%%%%%%%%%%%%%%%%%%%%%%%%%%%%%
\subsubsection{Tempered distortion and weak Gibbs}

In what follows we recall the notions of observables with tempered distortion and the notion of
weak Gibbs measures for a flow.

\begin{definition}\label{def:tempered}
Let $(X_t)_{t\in \mathbb R}$ be a continuous flow on a metric space $M$.
We say that a continuous observable $\psi : M \rightarrow \R$ has \emph{tempered variation} if
there is $\delta > 0$ such that $\lim_{t \to \infty}\frac{1}{t}\gamma(\psi , t, \delta) = 0$, where
$$
\gamma(\psi , t,\delta ):= \sup_{y \in B(x , t, \delta)} \Big|\int_{0}^{t}\psi(X_{s}(x)) - \psi(X_{s}(y)) \; ds\Big|.
$$
\end{definition}

\begin{definition}
Given a potential $\phi : M \rightarrow \R$ and a probability $\mu$, we say that $\mu$ is \emph{weak Gibbs}
with respect to $\phi$, with constant $P_{\mu} \in \mathbb{R}$, if for any $\vep > 0$ there exists $K_{t}(\vep)$ (depending only of $\vep$ and of the time $t$) so that $\lim_{t \to \infty}\frac{1}{t}\log K_{t}(\vep) = 0$ and
\begin{equation}\label{eq:Gibbss}
\frac{1}{K_{t}(\vep)} \leq \frac{\mu(B(x, t, \vep))}{\exp\Big[\int_{0}^{t}\phi(X_{s}(x))ds - tP_{\mu}\Big]} \leq K_{t}(\vep)
\end{equation}
%where $P_{\mu} := h_{\mu} (X_1)+ \int\phi \, d\mu$.
for every $x \in M$ and $t\in \mathbb R$.
If $\mu$ is $(X_{t})_t$-invariant then $P_{\mu} = h_{\mu} (X_1)+ \int\phi \, d\mu$.
 If the constants $K_t$ can be taken constant independently of
the time $t$ then we say that $\mu$ is a \emph{Gibbs measure}.
\end{definition}

%%%%%%%%%%%%%%%%%%%%%%%%%%%%%%%%%%%%%%%%%%%%%%%
\subsection{Statement of the main results}
\label{subsec:statements}
We are now in a position to state our main results in which we consider three different directions:
(i) relation between the gluing orbit property and uniform hyperbolicity,
(ii) large deviations results for semiflows with the gluing orbit property, and
(iii) criteria for suspension semiflows to satisfy the gluing orbit properties.

%%%%%%%%%%%%%%
\subsubsection{Gluing orbit property from the robust and generic viewpoints}

Our purpose here is to compare the gluing orbit property and the specification property for flows.
Taking into account that $C^1$-robustness of the specification property implies on topologically mixing
and uniformly hyperbolic flows (c.f.~\cite{AST}) one could wonder if the $C^1$-robustness of the gluing orbit
property is equivalent to the latter one. First we relate this notions with uniform hyperbolicity.

\begin{maintheorem}\label{thm:robust}
Let $X \in \mathfrak{X}^1(M)$ be so that there exists a $C^1$-open open neighborhood $\mathcal U \subset \mathfrak{X}^1(M)$ of $X$ so that the flow $(Y_t)_{t\in \mathbb R}$ associated to a vector field
$Y\in \mathcal U$ satisfies the gluing orbit property. Then the vector field $X$ generates a robustly
transitive Anosov flow $(X_t)_{t\in \mathbb R}$.
\end{maintheorem}

The following is a direct consequence of the previous result, the stability of Anosov flows and that
$C^1$-robust specification implies topologically mixing Anosov flows (c.f.~\cite{AST}).

\begin{maincorollary}\label{cor:robust}
Let $X\in \mathfrak{X}^1(M)$. The following are equivalent:
\begin{enumerate}
\item $X$ generates a topologically mixing Anosov flow;
\item $X$ satisfies the $C^1$-robust specification property; and
\item $X$ satisfies $C^1$-robustly both the topologically mixing and gluing orbit properties.
\end{enumerate}
\end{maincorollary}

In view of Corollary~\ref{cor:robust} it is natural to ask whether every topologically mixing smooth flow
with the gluing orbit property satisfies the specification property. We believe such examples may exist
for topologically mixing flows obtained as suspension of beta maps but we do not
prove or use this fact here.
Finally, following the same lines as in the proof of \cite[Theorem~2.6]{AST} we can
also prove the following:

\begin{maintheorem}\label{thm:generic}
There exists a $C^1$-residual subset $\mathcal R \subset \mathfrak{X}^1(M)$ so that any
vector field $X\in \mathcal R$ satisfying the gluing orbit property generates a transitive Anosov flow.
\end{maintheorem}

%%%%%%%%%%%%%%%%%%%%%%%%%%%%%%%%%%%%%%%%%%%%%%
\subsubsection{Large deviations principles}

In what follows we will be mostly interested in obtaining lower
bound large deviation estimates for semiflows with the gluing orbit property, a problem that revealed
difficulties even for uniformly hyperbolic flows.
Indeed, although a large deviations principle holds for \emph{continuous observables} and Axiom A diffeomorphisms
using the specification property (see e.g. \cite{You90}) a counterpart for flows does not follow immediately for Axiom A flows since  typically the strategy for lower bound estimates involve some specification property
which occurs only among topologically mixing dynamics.
In the mid nineties, Waddington~\cite{Wad96} obtained, among other limit theorems, a large deviations
principle for weakly topologically mixing Anosov flows.
Here we prove
a level-1 large deviations principle for every basic piece for an \emph{Axiom A flow} and any
\emph{continuous} observable, which is a consequence of the following theorem and the existence of the semiconjugacy to
symbolic dynamics obtained in \cite{BR75}.
Before stating it precisely just recall the topological pressure of the flow $(X_t)_t$ with respect to the potential $\phi$
is defined by
\begin{equation}\label{eq:equilibrium}
\Ptop((X_{t})_{t} , \phi)
	:= \sup_{\mu \in \mathcal M_{X_1}} \big\{ h_\mu(X_1) +\int \phi \, d\mu\big\}
\end{equation}
and an equilibrium state $\mu_\phi$ for $(X_t)_t$ with respect to the potential $\phi$ is a probability measure
that attains the supremum.

\begin{maintheorem}\label{thm:Wad}
Let $\sigma: \Sigma \to \Sigma$ be a subshift of finite type, $\rho :\Sigma \rightarrow \R$ be a H\"older continuous roof function
and $(X_{t})_{t}$ be the suspension flow associated to $\sigma$ and $\rho$. Let
$\phi: \Sigma_{\rho}\to \mathbb R$ be a continuous
potential so that $\mu_{\phi}$ is an unique equilibrium state for $(X_{t})_{t}$ with respect to $\phi$
and is a Gibbs measure.
For any \emph{continuous} observable $\psi: \Sigma_{\rho}\to\mathbb R$
it holds that
\begin{align*}
\limsup_{t\to +\infty} \frac1t \log \mu_\phi \Big( x\in \Sigma_{\rho} : \frac{1}{t}\int_{0}^{t} \psi(X_s(x)) \, ds \in [a , b] \Big)
    \le -\inf_{s \in [a , b]} I(s)
\end{align*}
and
\begin{align*}
\liminf_{t\to +\infty} \frac1t \log \mu_\phi \Big( x\in \Sigma_{\rho} : \frac{1}{t}\int_{0}^{t} \psi(X_s(x)) \, ds \in (a , b) \Big)
    \ge -\inf_{s \in (a , b)} I(s)
\end{align*}
where $I(s)=\sup\{ \Ptop((X_{t})_{t} , \phi) - \frac{h_\eta(\sigma)}{\int\rho d\eta} - \frac{\int \overline{\phi} \, d\eta}{\int\rho d\eta} : \eta \in \mathcal M_\sigma
 \,\&\, \frac{\int \overline{\psi} \,d\eta}{\int\rho d\eta} = s \}$ is the rate function and $\mathcal M_\sigma$ denotes the space of $\sigma$-invariant probability measures.
 %the observables $\bar\phi, \bar\psi $ are defined by $\bar\phi(x) = \int_0^{\rho(x)} \phi(X_s(x)) \, ds$
% and $\bar\psi(x) = \int_0^{\rho(x)} \psi(X_s(x)) \, ds$ respectively.
In particular, if $\bar \psi$ is not cohomologous to constant (meaning $\cM_\sigma \ni \eta \mapsto \int \bar \psi d\eta$ is not
a constant function) and the interval $[a , b]$
 does not contain $\int \psi \, d\mu_{\phi}$ then the right hand sides above are strictly negative.
% and thus there is exponential convergence.
\end{maintheorem}

Let us stress that large deviations lower bounds are much harder to obtain in virtue of the fact that points that are not
fastly converging to the mean can generate invariant measures that are not ergodic. It is at this point that some
specification-like property is needed. %in order to prove a large deviations principle for the flow.
The following result strenghs the proof of usual large deviations lower bounds requiring only the gluing orbit property.

\begin{maintheorem}\label{thm:LB}
Let $M$ be a metric space and $(X_{t})_{t}$ be a semiflow satisfying the gluing orbit property.
Assume $\phi : M \rightarrow \R$ is a bounded potential with tempered variation, $\mu$ is a weak Gibbs probability
with respect the $X_{t}$ and $\phi$ with constant $P=P_{\mu}$. Given real numbers $a\le b$:
 \begin{itemize}
 \item[i)] if $\psi : M \rightarrow \R$ is a bounded observable with tempered variation
then
 \begin{align*}
 \liminf_{t\to \infty}
 	& \frac{1}{t}\log \mu\Big(x \in M : \frac{1}{t}\int_{0}^{t}\psi(X_{s}(x))ds \in (a , b) \Big) \\
	 & \geq
 	- \inf\Big\{P_\mu - h_{\nu}(X_1) - \int\phi \, d\nu  : \nu \text{ is } X_{1}\text{-invariant and }
	 \int \psi \, d\nu \in (a , b) \Big\} \\
	 & \geq
	- \inf\Big\{P_\mu - h_{\nu}(X_1) - \int\phi \, d\nu  : \nu \text{ is } (X_{t})_t\text{-invariant and }
	 \int \psi \, d\nu \in (a , b) \Big\}
 \end{align*}
\item[ii)] if $M$ is compact and $\psi : M \rightarrow \R$ is continuous
then
 \begin{align*}
 \limsup_{t\to \infty}
 	& \frac{1}{t}\log \mu\Big(x \in M : \frac{1}{t}\int_{0}^{t}\psi(X_{s}(x))ds \in [a , b] \Big) \\
	 & \le
 	- \inf\Big\{P_\mu - h_{\nu}(X_1) - \int\phi \, d\nu  : \nu \text{ is } (X_{t})_t\text{-invariant and }
	 \int \psi \, d\nu \in [a , b] \Big\}.
 \end{align*}
\end{itemize}
\end{maintheorem}

In fact we can obtain lower bounds for the velocity of convergence of empirical measures to open sets in
the space of all probability measures. More precisely,
%\margem{Thiago: como a principio nao temos a estimativa superior em geral nao podmeos palicar o principio da ocntracao, %sendo assim acho melhor deixar um teorema separado para nivel 2}}
%We also get lower estimates of level-2 large deviations.
\begin{maintheorem}\label{thm:LB2}
Let $(X_{t})_{t}$ be a semiflow on a compact metric space $M$ having the gluing orbit property, $\phi : M \rightarrow \R$
be a bounded potential with tempered variation and $\mu$ be a weak Gibbs probability for $X_{t}$ with respect to $\phi$ with constant $P=P_{\mu}$. If $\psi : M \rightarrow \R$ is a bounded observable with tempered variation then
 \begin{align*}
 \liminf_{t\to \infty}
 	& \frac{1}{t}\log \mu\Big(x \in M : \frac{1}{t}\int_{0}^{t}\delta_{X_{s}(x)}ds \in V \Big) \\
	 & \geq
 	- \inf\Big\{P_\mu - h_{\nu}(X_1) - \int\phi \, d\nu  : \nu \text{ is } X_{1}\text{-invariant and }
	 \nu \in V \Big\}
 \end{align*}
for any open set $V$ in the space of probability measures on $M$.
\end{maintheorem}

%\begin{remark}
%\end{remark}

\begin{remark}
Arguments similar to the ones involved in the proof of the previous theorem yield
a large deviations principle holds for weak Gibbs measures, bounded observables with tempered variation
and discrete time maps with the gluying orbit property, extending \cite{You90}.
\end{remark}

%%%%%%%%%%%%%%%%%%%%%%%%%%%%%%%%%%%%%%%%%%%%%%%
\subsubsection{Criteria for gluing orbit properties}

In this subsection we provide some criteria for suspension flows to satisfy either the (uniform) gluing orbit property introduced in
Subsection~\ref{d.strong.specification} or a non-uniform measure theoretical gluing orbit property.

\begin{maintheorem}\label{thm:gluingAA}
Let $M$ be a metric space and let $f : M \rightarrow M$ satisfy the %specification property.
gluing orbit property.
Assume the roof function
$\rho: M \to \mathbb R_0^+$ is bounded from above and below, is uniformly
continuous
%or $\log$-H\"older,
and the constants
\begin{equation}\label{eq:distort}
C_\xi := \sup_{n\ge 1}
        \sup_{y\in B(x,n, \xi)}  |S_n r(x) -S_n r(y)|
         < \infty
         \quad\text{satisfy }
   	\lim_{\xi \to 0} C_\xi =0,
\end{equation}
where $S_n r=\sum_{j=0}^{n-1} r\circ f^j$.
Then the suspension semiflow $(X_t)_t$ has the gluing orbit property.
\end{maintheorem}

Let us observe that condition~\eqref{eq:distort} is a bounded distortion property for the roof function.
It is not hard to check It holds e.g. for H\"older continuous observables and uniformly expanding dynamics.
%\begin{remark}
%The bounded distortion property \eqref{eq:distort} holds e.g. in the case that $f$ is uniformly hyperbolic
%and $\rho$ is H\"older continuous. Indeed, for any $\xi>0$
%\begin{align*}
%|\sum_{j=0}^{n-1}  \rho(f^j(x)) - \sum_{j=0}^{n-1}  \rho(f^j(y))|
     %%\le \sum_{j=0}^{n-1}  | r(f^j(x)) -  r(f^j(y))|
    % \le \|\rho\|_\al \; \sum_{j=0}^{n-1} d(f^j(x), f^j(y))^\alpha
  %  < \infty
%\end{align*}
%for any $y\in B(x,n, \xi)$ and $n\ge 1$.
%\end{remark}
Since the requirement of the theorem on the base dynamics to satisfy a gluing orbit property then the later result
applies for suspension flows of transitive but non topologically mixing subshifts of finite type.
%In what follows, the latter property does not imply that the dynamical system to be topological mixing
%since, as a consequence of Theorem~\ref{thm:gluingAA} and the existence of a finite Markov partition for Axiom A
%lows, it is satisfied by all transitive (eventually non-mixing) Axiom A flows.
%
%
%
%
From the measure theoretical sense the shadowing of pieces of orbits can be actually non-uniform
in the following sense.

\begin{definition}\label{def:nugluing} (Non-uniform gluing)
Let $(X_t)_t$ be a semiflow on a separable metric space $M$ and consider a
$(X_t)_t$-invariant and ergodic probability measure $\overline \mu$.
We say that $((X_t)_t, \overline{\mu})$ has the  \emph{non-uniform gluing orbit property}
if for any $\vep>0$ and for $\overline \mu$-almost every  point $x \in M$ and $t\ge 0$
there exists
$T((x,t,\vep)>0$ so that
$$
\lim_{\vep\to 0} \limsup_{t\to +\infty} \frac{T(x, t,\vep)}{t}=0
$$
and for
$\overline \mu^k$-almost every points  $(x_1 , x_2, \dots, x_k) \in M^k$ and
times $t_1, \dots, t_k \ge 0$ there are $0\le p_i  \le T(x_i, t_{i}, \vep)$ and $x\in M$ satisfying
$$
d( X_t(x), X_t(x_1) ) <\vep
    \quad \forall t \in [0, t_1]
$$
and, if $\underline{x}_i= X_{\sum_{j=0}^{i-1} p_j+ t_j}(y) \in M$ then
$
d( X_{t}  (\underline{x}_i) , X_t(x_i) ) <\vep,
     \forall t \in [0, t_i]
$
for every $2 \le i \le k$.
\end{definition}

The previous property, similar to the gluing orbit property, roughly means that at least for
a full measure set of points (with respect to $\overline\mu$)
one can shadow the prescribed pieces of orbits by a real orbit and that the time length needed from one piece to the following can be bounded
by some time $T(x, t, \vep)$ that depends both on the point $x$ and the proximity $\vep$
but that sublinear growth in $t$. Actually the integrability of the roof function is enough to
obtain the non-uniform gluing orbit property.
This allows to consider e.g. suspension flows over subshifts
of countable type (see Section~\ref{sec:examples}).

\begin{maintheorem}\label{thm:nugluing1}
Let $M$ be a %(not necessarily compact)
metric space and assume that $f: M \to M$ satisfies the
gluing orbit property and let $\mu$ be an $f$-invariant, ergodic probability measure.
Assume the roof function $\rho: M \to \mathbb R_0^+$  is  continuous, bounded from below, $\rho\in L^1(\mu)$
and the constants
\begin{equation}\label{eq:distortnun}
C_\xi (x):= \sup_{n\ge 1}
        \sup_{y\in B(x,n, \xi)}  |S_n r(x) -S_n r(y)|
         < \infty
         \quad\text{satisfy}\quad
         \lim_{\xi \to 0} C_\xi(x) =0
         \text{ for $\mu$-a.e. $x$.}
\end{equation}
Then  the suspension flow $(X_t)_t$ has the non-uniform gluing orbit property with
respect to the invariant measure $\overline \mu$.
\end{maintheorem}

The previous result clearly applies in the case that $f$ is a countable full branch Markov expanding map
and any integrable roof function. Finally we prove the following:

\begin{maintheorem}\label{thm:nugluing2}
Let $M$ be a compact Riemannian manifold and let   $f: M \setminus \cC \to M$ be a $C^{1+\al}$ local
diffeomorphism in the whole manifold $M$ except in a non-degenerate critical/singular set $\mathcal{C}\subset {M}$:
there exists $B>0$ such that
\begin{enumerate}
\item    \quad $\displaystyle{\frac{1}{B}dist(x,\mathcal{C})^{\beta}\leq
    \frac {\|Df(x)v\|}{\|v\|}\leq B\,dist(x,\mathcal{C})^{-\beta} }$ for
    all $v\in T_x {M}$.
\item \quad For every $x,y\in {M}\setminus\mathcal{C}$ with
    $dist(x,y)<dist(x,\mathcal{C})/2$ we have
    $$\displaystyle{\left|\log\|Df(x)^{-1}\|- \log\|Df(y)^{-1}\|\:\right|\leq
    \frac{B}{dist(x,\mathcal{C})^{\beta}}dist(x,y)}.$$
\end{enumerate}
 Assume that $\mu$ is an $f$-invariant, ergodic and expanding measure
% hence satisfies the non-uniform specification property of Oliveira and Tian
and that the roof function $\rho : M \setminus \mathcal{C} \rightarrow \mathbb R_0^+$ is continuous, bounded from below, $\rho\in L^1(\mu)$
and the bounded distortion
condition \eqref{eq:distortnun} holds. Then the suspension flow
$(X_t)_t$ has the non-uniform gluing orbit property with respect to the invariant measure
$\overline \mu$.
\end{maintheorem}

The fundamental property used in the proof of the previous theorem is the non-uniform specification property
for the invariant measure. Although it is enough to assume the measure to satisfy the non-uniform gluing orbit
property we did not state the theorem in such abstract context due to the lack of motivating examples.
Thus, an analogous statement is most likely to hold whenever $f$ is a $C^{1+\alpha}$-diffeomorphism and $\mu$ is an
$f$-invariant hyperbolic measure.

%%%%%%%%%%%%%%%%%%%%
\section{Some examples}\label{sec:examples}

In this section we discuss the gluing orbit properties for some classes in both the discrete and the continuous time setting.
First we prove that  every transitive subshift of finite type satisfies the gluying orbit property.

\begin{example}\label{ex:1}
Given $d\ge 1$ and a transition matrix $A \in M_{d\times d}(\{0,1\})$ consider the one-sided subshift of finite type
 $\sigma : \Sigma_A \to \Sigma_A$ where
 $
 \Sigma_A =\{ (x_n)_{n\in \mathbb N_0} \in \{1, \dots, d\}^{\mathbb N_0} : A_{x_n,x_{n+1}=1}\}
 $
is endowed with the pseudo-distance
$$
d((x_n)_n, (y_n)_n)= 2^{-N}, \quad \text{where} \; N=\min \{ n\ge 0 : x_n \neq y_n\}.
$$
and let $\mathcal P$ denote the natural partition of $\Sigma_A$ in cylinders of size one.
Given $\vep>0$ let $N_\vep\ge 1$ be the smallest positive integer so that $2^{-N_\vep} <\vep$ and
consider the partition $\mathcal Q_\vep= \mathcal P^{(N_\vep)}$ where $\mathcal P^{(n)}=\bigvee_{j=0}^{n-1} \sigma^{-j}(\cP)$
is the dynamically defined partition. If $\cQ_\vep^{(n)}(x)$ denotes the element of the partition
$\cQ_\vep^{(n)}=\bigvee_{j=0}^{n-1} \sigma^{-j}(\cQ_\vep)$ that contains the point $x$ then for all our purposes
the dynamical ball $B_d(x,n,\vep)$ can be replaced by the partition element $\cQ_\vep^{(n)}(x)$.
We claim that if $\sigma : \Sigma_A \to \Sigma_A $ is transitive then it satisfies the gluing orbit property.
Recall that $\sigma : \Sigma_A \to \Sigma_A$ is transitive if and only if for any $i,j\in\{1, \dots, d\}$
there exists $n=n_{i,j} \ge 1$ so that $A^n_{i,j}=1$, where $A^n=(A^n_{i,j})_{i,j=1 \dots d}$.
Let $\tilde N= \max \{ n_{i,j} : i,j=1\dots d\} $.
Given $\vep>0$ take $p(\vep)= \tilde N + N_\vep$. Given $x_1, \dots, x_k \in \Sigma_A$ and $n_1, \dots, n_k \ge 1$
then it follows from the Markov property for $\sigma$ that $\sigma^{n_i}(\cQ_\vep^{(n_i)}(x_i))=\cQ_\vep(\sigma^{n_i}(x_i))$
for every $i=1\dots k$.
Set $P_i:=\sigma^{N_\vep} (\cQ_\vep(\sigma^{n_i}(x_i)))\in \cP$ and let $\hat P_{i+1}\in \cP$ denote the element of
the partition $\cP$ containing $x_{i+1}$. Using that $\cQ_\vep(x_{i+1}) \subset \hat P_{i+1} \in \cP$, by transitivity of $\sigma$,
there exists
$1\le p_i \le \tilde N$ so that
$
\sigma^{p_i} ( P_i )
	\supset P_{i+1}
	\supset \cQ_\vep(x_{i+1})
$
for every $1=1\dots k-1$.
This proves the gluing orbit property for $\sigma_A$ as claimed.
%For $x,y \in \Sigma_A$ arbitrary $\sigma^N(\cP^{(N)}(x))=\cP(\sigma^N(x))$ and $\cP^{(N)}(y) \subset \cP(y)$.
%Thus it is enough to prove there exists $\tilde N\ge 1$ so that that for any $P,Q\in \cP$ there exists $1\le n \le \tilde N$
%so that $\sigma^n (P) \cap Q \neq \emptyset$ (since the later implies $\sigma^n (P) \supset Q$).
%Since the cardinality of the partition $\cP$ is at most $k$, then the existence such $N$ follows by
%transitiveness of $(\sigma, \Sigma_A)$.
\end{example}

Indeed,
the previous example can be adapted to deal with subshifts of countable type $\sigma : \Sigma \to \Sigma$
with $\Sigma \subset S^{\mathbb N}$ and an infinite set $S\subset \mathbb N$. These model many non-uniformly
hyperbolic dynamical systems. If $\Sigma=S^{\mathbb N}$ is
the full shift then it is clear it satisfies the specification property. The same arguments as the ones of the previous example
yield that subshifts of countable type with the gluing orbit property also include important classes of subshifts as the
ones with the so called big image and preimage property (see e.g. ~\cite{MUr}).
%\begin{example}
%Let $\{1, \ldots, k, \ldots\}^{\mathbb{N}_{0}}$ be the enumerable shift space endowed with the pseudo-distance
%$d((x_n)_n, (y_n)_n):= k^{-N}$, where $N=\min \{ n\ge 0 : x_n \neq y_n\}$. Let $\Sigma_{A} \subset \{1, \ldots, k, \ldots\}^{\mathbb{N}_{0}}$ be a invariant subset by shift $\sigma$ with big images property. Thus $\sigma : \Sigma_{A} \rightarrow \Sigma_{A}$ has the uniform gluing orbit property. In particular, the enumerable full shift has the uniform gluing orbit property.
%\end{example}

\begin{example}
Let $M$ be a compact Riemannian manifold and $\Lambda \subset M$ be a transitive hyperbolic set for a $C^1$ flow $(X_t)_t$.
We notice that, via the existence of Markov partitions (see e.g. \cite{BR75,Bo75}), the restriction of the flow $(X_t)_t$ to
$\Lambda$ is semiconjugated to suspension flow with over a transitive subshit of finite type $\sigma$ and a H\"older
continuous roof function $\rho$ bounded away from zero. Since $\sigma$ satisfies the gluing orbit property (c.f.
 Example~\ref{ex:1}) and every H\"older observable on the shift satisfies the bounded distortion condition \eqref{eq:distort}
it follows from Theorem~\ref{thm:gluingAA} that $\Lambda$ has the gluing orbit property.
Theorem~\ref{thm:Wad} yields large deviations principles for the flow with respect to all continuous observables.
Theorem~\ref{thm:LB2} implies on a level-2 large deviations lower bound for hyperbolic flows.
\end{example}

Let us observe that suspension flows over subshifts of countable type, since do not have a compact phase space, are not
expected to have the gluing orbit property in general. Theorem~\ref{thm:nugluing1} implies that the non-uniform gluing orbit
property holds provided the roof function is integrable and satisfies the distortion condition~\eqref{eq:distortnun}.

\begin{example}\label{ex:geo}
It is well known from the pioneering works of Anosov and Sinai $C^2$-Riemannian metrics with strictly negative curvature
generate Anosov geodesic flows \cite{Anosov,AnosovS}, hence satisfy the gluing orbit property restricted to every transitive subset of the non-wandering set. In the case of non-strictly negative curvature a partial solution has been recently announced by Burns, Climenhaga, Fisher and Thompson~\cite{BCFT}. Bessa, Torres and Varandas~\cite{BTV} announced recently that
there exists a residual subset of $C^1$-metrics with bounded curvature whose geodesic flow satisfies a reparametrized gluing orbit property: for any $\epsilon>0$ there exists $K=K(\epsilon) \in \mathbb R^+$ such that
for any points $x_1, x_2, \dots, x_k \in M$ and times $t_1, \dots, t_k \ge 0$
there are $p_1, \dots, p_k  \le K(\epsilon)$, a reparametrization $\tau \in \mbox{Rep}(\epsilon)$ and a point $y\in M$
so that
$$
d( X^{\tau(t)}(y)), X^t(x_1) ) <\epsilon
    \quad \forall t \in [0, t_1]
$$
and
$$
d( X^{\tau(t+\sum_{j=0}^{i-1} p_j+ t_j)}(y) , X^t(x_i) ) <\epsilon
    \quad \forall t \in [0, t_i]
$$
for every $2 \le i \le k$.
By $\mbox{Rep}$ we denote the set of all increasing homemorphisms $\tau\colon \R \rightarrow \R$,
called {\em reparametrizations}, satisfying $\tau(0)=0$. Fixing $\epsilon>0$, we define the set
$$\mbox{Rep}(\epsilon)=\left\{\tau \in \mbox{Rep}: \left|\frac{\tau(t)}{t}-1 \right|<\epsilon, \, t \in \R \right\},$$ of
the reparametrizations $\epsilon$-close to the identity.
Let us remark that the reparametrization $\tau$ above satisfies
$
\tau(t_1+p)-\tau(t_1)
	\le (1+\epsilon) p
	\le (1+\epsilon) K(\epsilon).
$
Hence, the later condition is substantially weaker than specification (since it does not imply topologically mixing)
but implies strong transitivity conditions: for any two balls of radius $\epsilon$ there exists a point whose piece of orbit
up to a definite time $(1+\epsilon) K(\epsilon)$ (depending only on $\epsilon$) intersects both balls.
%\color{red}{Throughout we shall
%refer to this as the \emph{strong transitivity} condition.}
\end{example}

In the following example we shall consider flows with an intermittency phenomenon.

\begin{example}
Consider $M=[0,1]$ and the Maneville-Pomeau map $f_\alpha:[0,1]\to [0,1]$ given by
\begin{equation*}\label{eq. Manneville-Pomeau}
f_\al(x)= \left\{
\begin{array}{cl}
x(1+2^{\alpha} x^{\alpha}) & \mbox{if}\; 0 \leq x \leq \frac{1}{2}  \\
2x-1 & \mbox{if}\; \frac{1}{2} < x \leq 1.
\end{array}
\right.
\end{equation*}
for $\alpha \in (0,1)$. Since this map is semiconjugated to the full shift on two symbols
then it satisfies the specification property.
For any roof function $\rho$ satisfying \eqref{eq:distort} and bounded away from zero the semiflow
has the gluing orbit property.

Take $\phi : M_\rho \to \mathbb R$ smooth observable and the reduced observable $\bar \phi : M \to \mathbb R$
given by $\bar\phi(x) = \int_0^{\rho(x)} \phi( X_s(x,0)) \, ds$. If $\bar\psi$ satisfies $\sup \bar \phi -\inf \bar \phi  < \log 2$
there exists a unique equilibrium state $\mu_{\bar\phi}$ for $f$ with respect to $\bar\phi$ (see e.g.~\cite{VV10}).
Furthermore, the unique equilibrium state $\mu:= \mu_\phi \times \Leb / \int \rho \, d\mu_{\phi}$ for the flow
satisfies a large deviations principle for every continuous observable.
This is the case e.g. for the potential $\phi=0$ and the corresponding (unique) maximal entropy measure $\mu_0$.
%Using that for any smooth roof function $r: [0,1] \to \mathbb R^+$
%it holds that
%\begin{align*}
%\sup_{y\in B(x,n, \xi)}  |S_n r(x) -S_n r(y)|
 %	& \le  \sup_{y\in B(x,n, \xi)}  \sum_{j=0}^{n-1} |r\circ f_\al^j(z) - r\circ f_\al^j(y)| \\
    %      	& \le  \|Dr \|_\infty \; \sup_{y\in B(x,n, \xi)}  \sum_{j=0}^{n-1}  \, d(f_\al^j(x), f_\al^j(y)) \\
	%          & \le  C \|Dr \|_\infty \; \sup_{y\in B(x,n, \xi)}  \sum_{j=0}^{n-1} \, \frac{1}{Df_\al^{n-j}(f_\al^j(x))}
%\end{align*}
%which tends to zero as $\xi \to 0$ (hence satisfies the bounded distortion condition~\eqref{eq:distort}),
%the corresponding suspension flow satisfies the gluing orbit property.
%
%Thus, Theorem~\ref{thm:LB} implies on a large deviations principle for every observables with tempered variation.
In the case there are more than one equilibrium state the rate function in the large deviations principle may fail to be
strictly convex, in which case the exponential large deviations can fail.
For instance, Melbourne and Nicol \cite{MN08} obtained (upper and lower) polynomial deviation
bounds for H\"older continuous observables and the SRB measure of these suspension semiflows.
\end{example}

It is likely that the previous example can be adapted to deal with more general
almost-hyperbolic flows (e.g. suspension flows of diffeomorhisms obtained from Anosov
diffeomorphisms by isotopy to obtain finitely many indifferent periodic points as in \cite{HY}).

\section{Some comments and open questions}\label{open}

After introducing this property of gluing, it seems natural not only to verify other examples that do satisfy it
but also to explore it as a tool. Similarly to the use of specification as a tool, we expect the gluing orbit property to be
an useful tool to derive other applications (e.g. multifractal analysis). Let us also stress that the proof of Theorem~\ref{thm:gluingAA} in the stronger context of
a bi-Lipschitz homeomorphism  $f$ and H\"older continuous roof function $\rho$ can be slightly simplified. This follows
from the fact that, under these stronger assumptions, one may make use of the pseudo-metric $d_{\pi}$ instead of
the Bowen-Walters distance. %We refer the reader to~\cite[Appendix]{BS} for an account on this relation.
%\vspace{.1cm}
Although the gluing orbit property is strictly weaker than the specification property it is an interesting challenge to
study their relation. With that purpose we pose the following question:

\vspace{.1cm}
\noindent {\bf Question 1:} Let $\mathcal G$ be the class of $C^1$-diffeomorphisms with the gluing orbit property.
Is there a topologically large (e.g. open, dense, residual, ..) subset $\mathcal G_1$ of $\mathcal G$ so that every
topologically mixing diffeomorphism in $\mathcal G_1$ satisfies the specification property?

\vspace{.1cm}
We believe some regularity (e.g. smoothness) of the dynamical system should be necessary for presenting a positve
answer to the later question.
The results by Bowen~\cite{Bo74} and Haydn and Ruelle~\cite{Ru92,HR} on the thermodynamical formalism of expansive maps with the specification property and recent extensions by Climenhaga and Thompson~\cite{CT13} motivate the study of
the ergodic features of maps with the gluing orbit property.

\vspace{.1cm}
\noindent {\bf Question 2:} Let $f$ be an expansive map (diffeomorphism or non-critical endomorphism) with the gluing orbit property. Does there exist a finite number of equilibrium states for every regular (e.g. H\"older continuous) potential? Do these
have exponential decay of correlations? The associated transfer operator is quasi-compact on the $L^{p}$ spaces?

\vspace{.1cm}
In the discrete time setting one could hope to obtain a spectral decomposition of the non-wandering set in a finite
number of pieces, similar to the one for hyperbolic dynamics, that could guarantee that some power of the dynamics
satisfies the specification property for each transitive piece in the decomposition. Since constant reparametrizations
of time-continuous dynamics does not change the mixing properties this picture cannot be expected for flows with
the gluing orbit property. Some interesting classes of dynamical systems for which decay of correlations and large deviations
that still remain not fully understood are billiards and geodesic flows. In virtue of our large deviations results
it is natural to ask the following questions:

\vspace{.1cm}
\noindent {\bf Question 3:}
(a) Which billiard flows satisfy the gluing orbit property? Do these include dispersing or Sinai billiards flows?
(b) Do ``most" geodesic flows satisfy the gluing orbit property?
\vspace{.1cm}

By Example~\ref{ex:geo} the answer to item (b) in the previous question has partial answers in either lower
topologies or whenever some condition is given on the set of points with non-negative curvature.
We stress that the notions of non-uniformly gluing and a similar notions of almost gluying (similar to the similar notion from \cite{RVZ}) can probably be used to study large deviations and multifractal analysis (see e.g. \cite{BoVa}).
Finally, it is well known from earlier work of Sigmund's~\cite{Sig}
for maps with specification have a rich simplex of invariant probability measures. We refer the reader to the survey by
Kwietniak, Lacka and Oprocha~\cite{KLO} for a good account on some recent developments and the study of this
simplex for maps with
specification like properties. Taking this into account it is natural to ask the following question:

\vspace{.1cm}
\noindent {\bf Question 4:}
What is the ``richness" of the simplex of invariant probability measures for dynamics with the gluing orbit property?
Which items of Sigmund's theorem (c.f. Theorem~11 in \cite{KLO}) still hold for dynamics with the gluing orbit property?
\vspace{.1cm}

%{Questao: existe condição de perron que garanta, como em \cite{RY08}, desvios exponenciais
%para fluxos? Qual a condição boa para diferenciabilidade da pressão?

%%%%%%%%%%%%%%%%%%%%%%%%%%%%
\section{The gluing orbit property and uniform hyperbolicity}\label{proofs}

%%%%%%%%%%%%%%%%%%%%%%%%%%%%
\subsection{Proof of Theorems~\ref{thm:robust} and ~\ref{thm:generic}}\label{subset:hyper}

In this section we shall prove that either $C^1$-robustly or $C^1$-generically, the gluing orbit property implies the flow to be
uniformly hyperbolic. The proofs here follow closely the strategy in \cite{AST}
of proving that the later conditions imply that the flow is a star flow, a condition that is equivalent to uniform hyperbolicity of the
flow in the $C^1$-topology (we refer the reader to the subsections below for details).
The main novelty is to understand how the gluing orbit property can be used to
establish the constancy of index among hyperbolic critical elements (c.f. Proposition~\ref{prop:index} below).

%%%%%%%%%%%%%%%%%%%%%%%%%%%%
\subsubsection*{Proof of Theorem~\ref{thm:robust}}

Our purpose here is to prove that the $C^1$-robustness of the gluing orbit property implies
on the uniform hyperbolicity of the original flow.
The argument follows along the same lines of the strategy to prove that robust specification implies
on uniform hyperbolicity, with some extra effort due to the fact that one cannot a priori choose a
definite iterate of the flow for which stable and unstable manifolds are long enough to intersect.
One key ingredient is to prove that all hyperbolic critical elements are necessarily of the same index, that is,
the dimension of its stable bundle in the hyperbolic decomposition (this is the counterpart of
\cite[Theorem~3.3]{AST} in our setting).

Let us introduce some necessary notations. Given a hyperbolic critical element $p$ with hyperbolic decomposition
$T_{\mathcal O(p)} M = E^s \oplus < X > \oplus E^u$ (if $p$ is periodic) or $T_p M = E_p^s \oplus E_p^u$ (if $p$ is a singularity)
denote the stable index by $\text{ind}^s(p):=\dim E^s_p$. Given a hyperbolic critical element $p$ and $\vep>0$,
the local strong stable manifolds of size $\vep$ at $p$ is given by
$$
W^{ss}_\vep(p)
	= \{ x\in M \colon d( X_{t}(x), X_{t}(p)) \le \vep \; \text{for every } t \ge 0 \}
$$
is a smooth submanifold (well defined by uniform hyperbolicity) and set
$$
W^{cs}_\vep(\mathcal O(p)) = \bigcup_{t \in \mathbb R} W^{ss}_\vep(X_t(p)).
$$
%The local center-stable
%and is defined by
%$$
%W^{s}_\vep(p) = \bigcup_{|t|\le \vep} X_t (W^{ss}_\vep(X_t(p)))
%$$
The local strong unstable manifolds $W^{uu}_\vep(p)$ of size $\vep$ at $p$ and the submanifold
$W^{cs}_\vep(\mathcal O(p))$ are defined analogously by the corresponding stable manifolds for
the reversing time flow $(X_{-t})_t$.
%$$
%W^{ss}_\vep(p)
%	= \{ x\in M \colon d( X_{t}(x), X_{t}(p)) \le \vep \; \text{for every } t \ge 0 \}
%$$
%and
%$$
%W^{uu}_\vep(p)
%	= \{ x\in M \colon d( X_{-t}(x), X_{-t}(p)) \le \vep \; \text{for every } t \ge 0 \},
%$$
%respectively, are smooth submanifolds which are well defined by uniform hyperbolicity. The local center-stable
%and local center-unstable manifolds are defined by
%$$
%W^{s}_\vep(p) = \bigcup_{|t|\le \vep} X_t (W^{ss}_\vep(X_t(p)))
%	\quad\text{and}\quad
%	W^{u}_\vep(p) = \bigcup_{|t|\le \vep} X_t (W^{uu}_\vep(X_t(p)))
%$$
%respectively.

\begin{proposition}\label{prop:index}
If $p,q$ are hyperbolic critical elements for $X\in \mathfrak X^1(M)$ and the generated flow $(X_t)_t$
satisfies the $C^1$-robust gluing orbit property
then $\text{ind}^s(p) = \text{ind}^s(q)$. %$x$ and $y$ have the same index.
Moreover, for any $\vep>0$ there exists $L=L(\vep)>0$ so
that $X_L(W_\vep^{cu}(\mathcal{O}(p))) \cap W_\vep^{cs}(\mathcal{O}(q))) \neq \emptyset$ and
$X_L(W_\vep^{cu}(\mathcal{O}(q)))) \cap W_\vep^{cs}(\mathcal{O}(p))) \neq \emptyset$. In particular
$W^{cs}(p)$ and $W^{cu}(q)$ intersect.
\end{proposition}

\begin{proof}
Let $X\in \mathfrak X^1(M)$ satisfy the $C^1$-robust gluing orbit property and $p,q$ hyperbolic
critical elements for $X$. There are three cases to consider, depending on whether the critical elements
are periodic orbits  or singularities. We recall that the gluing orbit property implies transitivity and,
consequently, all periodic points and singularities are of saddle type.

Assume first that $p,q$ are hyperbolic periodic orbits.
Take $\vep>0$ and let $L(\vep)>0$ be given by the gluing orbit property.  Hence, for any
$t>0$ there are $0\le p_1(t)=p_1(t,p,q)  \le L(\vep)$ and $z_t=z(t,p,q) \in M$ so that
$$
d( X_{-s}(z_t)), X_{-s}(p) ) <\vep
	\quad\text{and}\quad
d( X_{s}  (X_{p_1(t)}(z_t)) , X_s(q) ) <\vep
$$
for every $s \in [0, t]$. By compactness of $[0,L(\vep)]$ one can take a subsequence $t_n\to\infty$
so that $p_1(t_n,p,q) \to \tilde p_1 \in [0,L(\vep)]$ as $n$ tends to infinite.
Up to consider a subsequence we may assume also that the sequence $(z(t_n,p,q))_{n\in \mathbb N}$ is convergent to some $z\in M$. This implies that
\begin{equation}\label{propeq1}
d( X_{-s}(z)), X_{-s}(p) ) \le \vep
	\quad\text{and}\quad
d( X_{s}  (X_{\tilde p_1}(z)) , X_s(q) ) \le \vep
\end{equation}
for every $s\in \mathbb R^+$, meaning that $z\in W^{cs}_\vep(\mathcal{O}(p)) \cap X_{\tilde p_1}(W^{cu}_\vep (\mathcal{O}(q))))$. Since  $0<\tilde p_1 \le L$ and
$ X_{\tilde p_1}(W^{cu}_\vep (\mathcal{O}(q))) \subset X_{L}(W^{cu}_\vep (\mathcal{O}(q)))$ this yields
$W^{cs}_\vep(\mathcal{O}(p)) \cap X_{L}(W^{cu}_\vep (\mathcal{O}(q)))\neq \emptyset$.
A similar argument (reverting the time) yields
$X_{L}(W^{cu}_\vep(\mathcal{O}(p))) \cap W^{cs}_\vep (\mathcal{O}(q)))\neq \emptyset$.

In the case that $p,q$ are both singularities then $\mathcal{O}(p)=p$ and $\mathcal{O}(q)=q$.
Proceeding as before we obtain as in the proof of ~\eqref{propeq1} we get that there exists $z\in M$
so that $d( X_{-s}(z)), p) <\vep$
and
$d( X_{s}  (X_{\tilde p_1}(z)) , q) <\vep$
for every $s\in \mathbb R^+$. This ultimately implies that $X_{L}(W_\vep^{uu}(p)) \cap W_\vep^{ss}(q) \neq \emptyset$.
Since $W_\vep^{cs}(p)=W_\vep^{ss}(p)$ and $W_\vep^{cu}(p)=W_\vep^{uu}(p)$, and analogous statements hold for $q$ then the
proposition follows in this second situation.

The proof of the proposition in the case that $p$ is a periodic orbit and $q$ is a singularity is completely
analogous to the previous ones and is left as an exercise to the reader.
\end{proof}

Now, to complete the proof of the theorem, assume that $X \in \mathfrak{X}^1(M)$ admits a $C^1$-open neighborhood
$\mathcal U \subset \mathfrak{X}^1(M)$ of vector fields $Y\in \mathcal U$ for which the
corresponding flows $(Y_t)_{t\in \mathbb R}$ satisfy the gluing orbit property.
Since every flow with the gluing orbit property is necessarily transitive then every $C^1$-vector field in
$\mathcal U$ generates a robustly transitive flow and so all periodic points and singularities are of saddle type.

It is well known that the set of Kupka-Smale flows (i.e. flows whose critical elements are hyperbolic and their stable and
unstable manifolds either do not intersect or intersect transversely) is $C^1$-generic in $\mathfrak{X}^1(M)$ (hence dense in
$\mathcal U$).
In particular, if $X \in \mathcal U$ is Kupka-Smale and $p,q$ are hyperbolic critical elements for $X$ such that
$\dim W^{cs}(p)+\dim W^{cu}(q) \le \dim M$ then $W^{cs}(p) \cap W^{cu}(q)=\emptyset$
(see Lemma~3.4 in ~\cite{AST}).
In view of Proposition~\ref{prop:index} the intersections $W^{cs}(p) \cap W^{cu}(q) \neq \emptyset$ are necessarily non-empty.
This implies hyperbolic singularities and hyperbolic periodic orbits for $X$ cannot coexist.

Since for $C^1$-generic vector fields the critical elements are dense (c.f. Pugh's general density theorem, see \cite{Pugh})
the critical elements of $X$ cannot be all singularities, since otherwise the vector field $X$ would be constant to zero, which
contradicts the robust transitiveness assumption. Thus $\text{Sing}(X)=\emptyset$ for any
$X\in \mathcal U$ and that the index of all hyperbolic periodic orbits is constant in a neighborhood of $X$.
We will make use of the following perturbation result.

\begin{lemma}\label{le:bifurca}
If $X\in \mathcal U$ and a periodic orbit of $X$ is not hyperbolic then there exists a $C^1$-arbitrarily close perturbation $Y\in \mathfrak{X}^1(M)$ displaying two hyperbolic periodic orbits of different index.
\end{lemma}

The proof of the previous lemma follows \emph{ipsis literis} the one of \cite[Theorem~4.3]{AST}
and relies on a version of Franks' lemma for flows (Lemma~1.3 in \cite{MSS}). Moreover, since the
robust weak specification assumption in \cite[Lemma~1.3]{MSS} is not used for the proof of the previous lemma
we shall omit its proof.
Now, since all hyperbolic periodic points for vector fields in $\mathcal U$ have the same index then it follows from
Lemma~\ref{le:bifurca} that every vector fields in $\mathcal U$ do not admit non-hyperbolic periodic points.
On the one hand, by Gan, Wen and Zhu \cite{GWZ}, every robustly transitive set which is strongly homogeneous of the same index is sectionally hyperbolic. On the
other hand, any sectionally hyperbolic flow without singularities is uniformly hyperbolic
(see~\cite{GW}). This implies that $X$ is a transitive Anosov flow and finishes the proof of
Theorem~\ref{thm:robust}.

 %%%%%%%%%%%%%%%%%%%%%%%%%%%%
 \subsubsection*{Proof of Theorem~\ref{thm:generic}}

We claim the existence of a $C^1$-residual subset $\mathcal R \subset \mathfrak{X}^1(M)$ so that any
 $X\in \mathcal R$ with the gluing orbit property generates an Anosov flow. Consider the
 $C^1$-residual subset $\mathcal R=\mathcal R_1 \cap \mathcal R_2$, where $\mathcal R_1$ denotes
 the  $C^1$-residual subset of Kupka-Smale vector fields and $\mathcal R_2$ denotes the $C^1$-residual subset given by Pugh's general density theorem. Since hyperbolic critical elements are dense and the index of all periodic points is constant (c.f. Proposition~\ref{prop:index}) then every $X \in \mathcal R$ admits no singularities.
We need the following auxiliary result.

\begin{lemma}\label{le:persistper} \cite[Lemma 5.1]{AST} There exists a residual subset $\mathcal{R}_3$ of $\mathfrak{X}^1(M)$ so that if $X\in \mathcal{R}_3$ is $C^1$-approximated by a sequence $(X_n)_n$ such that each
$X_n\in \mathfrak{X}^1(M)$ has two distinct hyperbolic periodic orbits, $p_n, q_n$ with different indices and
with $d(p_n,q_n)<\vep$, then there exist two distinct hyperbolic periodic points, $p,q$ for $X$ with
different indices and with $d(p,q)<2\vep$.
\end{lemma}

We claim that any $X \in \mathcal R \cap \mathcal R_3$
with the gluing orbit property generates a star flow, that is,
there exists an open neighborhood $\mathcal U$ of $X$ so that all critical elements of $Y\in \mathcal U$
are hyperbolic.  Assume, by contradiction, this is not the case. Then, there exists a sequence $X_n \to X$
(in the $C^1$-topology) and $x_n$ a non-hyperbolic critical element for the vector field $X_n$.
This, together with Lemma~\ref{le:bifurca}, implies that $X$ can be approximated bt a sequence $(\tilde X_n)_n$
of $C^1$-vector fields each of which exhibits a pair of periodic points $p_n,q_n$ with different index. By
Lemma~\ref{le:persistper}, $X$ has two periodic orbits of different index, which contradicts the fact that
all periodic points have the same index. This completes the proof of the theorem.

%%%%%%%%%%%%%%%%%%%%%%%%%%%%
\section{From gluing to large deviations}\label{sec:LDP}

This section is devoted to the proof of our large deviations results (Theorems~\ref{thm:LB}, ~\ref{thm:Wad} and
~\ref{thm:LB2}).

%%%%%%%%%%%%%%%%%%%%%%%%%%%%
\subsection{Reduction to the Poincar\'e map}

Given a suspension semiflow  $(X_{t})_{t\ge 0}$ over a base dynamics $f$ with roof function $\rho$, an $f$-invariant
probability measure $\mu$ and an observable $\psi : M_\rho \rightarrow \R$, consider the reduced observable
$\overline \psi : M \to \mathbb R$ given by
$
\overline{\psi}(x) := \int_{0}^{\rho(x)}\psi(X_{s}(x))ds
$
and the flow invariant probability measure  $\bar{\mu} := \frac{\mu \times \text{Leb}}{\int \rho d\mu}$.
The following lemma relates equilibrium states for $(X_t)_t$ with equilibrium states for $f$.

\begin{lemma}
Let $(X_t)_{t \ge 0}$ be a suspension semiflow over a continuous map $f: M \to M$ with a roof function
$\rho : M \to \mathbb R^+$ bounded away from zero. Given a potential $\phi: M_\rho \to \mathbb R$ the following
are equivalent:
\begin{itemize}
\item[(a)] %the probability measure
$\mu_\phi= \mu_f \times Leb / \int \rho \, d\mu_f$ is an equilibrium state for  $(X_t)_t$ with respect to $\phi$
\item[(b)] $\mu_f$ is an equilibrium state for $f$
with respect to the potential $\bar \phi - P \rho$
\end{itemize}
where %the potential $\bar \phi :  M\to \mathbb R$ is given by $\bar\phi(x) = \int_0^{\rho(x)} \phi(X_s(x)) \, ds$ and
$P=P(\phi)$ denotes the topological pressure of the flow with respect to $\phi$.
\end{lemma}

\begin{proof}
If $\mu_\phi$ is an equilibrium state for  $(X_t)_t$ with respect to $\phi$ it follows by equation~\eqref{eq:equilibrium} that
$$
h_{\mu_\phi} (X_1) + \int \phi \, d\mu_\phi
	= \sup_{\hat \eta \in \mathcal M_1 ((X_t)_t)}  \big\{  h_{\hat \eta} (X_1) + \int \phi \, d\hat\eta  \big\}
	=:P(\phi).
$$
Since $\rho$ is bounded away from zero there is a map between the space  $\{ \eta \in \mathcal{M}_\sigma :  \int \rho \, d\eta <\infty \}$ and the space of $(X_t)_t$ invariant probability measures via the map
$
\eta \mapsto \hat \eta:= \frac{(\eta \times \Leb)}{ \int \rho \, d\eta}.
$
It follows from a simple computation and the Abramov formula (see e.g.~\cite{Totoki}) that
\begin{equation}\label{induz}
\int \phi \, d\hat \eta = \frac{\int \bar \phi \, d\eta}{\int \rho \, d\eta}
	\quad\text{and}\quad
	h_{\hat\eta} (X_t) = \frac{|t| \, h_\eta (f)}{\int \rho \, d\eta}
\end{equation}
for every $t\in \mathbb R^+$ and every $(X_t)_t$-invariant probability measure $\hat \eta$. Thus, for any $(X_t)_t$
invariant probability measure $\hat\eta$ it holds that
\begin{align*}
0   \ge  - P(\phi) + h_{\hat \eta} (X_1) + \int \phi \, d\hat\eta
	 = \frac{- P(\phi) \, \int \rho \, d\eta + h_{\eta} (f) + \int \bar\phi \, d\eta }{ \int \rho \, d\eta}
\end{align*}
which is equivalent to the equation
$$
h_{\eta} (f) + \int (\bar\phi - P(\phi) \rho) \, d\eta  \le 0
$$
for every $f$-invariant probability measure $\eta$. Thus $\hat \eta$ is an equilibrium state for $(X_t)_t$ with respect to $\phi$
if and only if $\eta$ is an equilibrium state for $f$ with respect to $\phi - P(\phi) \rho$ and $\Ptop(f, \bar\phi - P(\phi) \rho)=0$.
This finishes the proof of the lemma.
\end{proof}

\begin{lemma}\label{lem1}
Let $(X_t)_{t\in \mathbb R^+}$ be a continuous semiflow a metric space $M$
and let $\psi : M \rightarrow \R$ be an observable.
Assume that either: (i) $M$ is compact and $\psi$ is continuous, or (ii) $\psi$ has tempered variation.
Given $(a , b) \subset \R$ and $\vep> 0$ there exists $\delta, t_{0}
> 0$ such that if  $\frac{1}{t}\int_{0}^{t}\psi(X_{s}(x))ds \in (a , b)$ and $t \geq t_{0}$
then
$$
\frac{1}{t}\int_{0}^{t}\psi(X_{s}(y)) \; ds \in (a - \vep , b +\vep)
	\quad\text{for every } y \in B(x , t, \delta).
$$
\end{lemma}

\begin{proof}
In case (i), since $\psi$ is continuous and $M$ is compact then it is
uniformly continuous. Given $\vep>0$ arbitrary let $\delta_0>0$ be such that $|\psi(x)-\psi(y)|<\vep$ for
every $y\in B(x,\delta_0)$. Thus, for any $t\ge 0$, $0<\delta<\delta_0$ and
$ y \in B(x , t, \delta)$ it holds that
$$
\left| \frac{1}{t}\int_{0}^{t}\psi(X_{s}(y))ds - \frac{1}{t}\int_{0}^{t}\psi(X_{s}(x))ds \right|
	 \le  \frac{1}{t}\int_{0}^{t} | \psi(X_{s}(y)) - \psi(X_{s}(x)) | ds
	 < \vep,
$$
which proves the lemma in this first case. In case (ii), since $\psi$ has tempered variation, for any $\vep>0$
there exists $\delta > 0$ and $t_{0} > 0$ large such that
$\Big|\frac{1}{t}\int_{0}^{t}\psi(X_{s}(y))ds - \frac{1}{t}\int_{0}^{t}\psi(X_{s}(x))ds \Big|  \leq \vep$ for every $t\ge t_0$
and $y\in B(x,t,\delta)$. The proof now follows analogously as before.
\end{proof}

If $M$ is a compact space, the space $\mathcal M(M)$ of probability measures on $M$ endowed with the weak$^*$-topology is a compact metrizable space. Given a countable and dense
subset $(g_i)_{i \in \mathbb N}$ of continuous observables with $\|g_i\|=1$ for every $i \in \mathbb N$ consider
the metric $\tilde d$ on $\mathcal M(M)$ given by
\begin{equation*}\label{eq:metrizable}
\tilde d (\eta_1, \eta_2)
	:= \sum_{i \in \mathbb N} \frac{1}{2^i} \big|  \int g_i \, d\eta_1 - \int g_i \, d\eta_2 \big|.
\end{equation*}
Observe that $\tilde d$ is invariant by translation (i.e. $\tilde d (\eta_1+\eta_3, \eta_2+\eta_3) = \tilde d (\eta_1, \eta_2)$ for all
probabilities $\eta_1,\eta_2, \eta_3$) and that the function $\tilde d(\cdot, \eta)$ is convex
for any fixed probability measure $\eta$.

\begin{lemma}\label{lem2}
Let $(X_t)_{t\in \mathbb R}$ be a continuous flow on a compact metric space $M$
and let $\tilde d$ be the previously defined metric. %\eqref{eq:metrizable} %of the weak$^{\star}-$topology on the space of probability measures.
Given $\vep> 0$ there exists $\delta > 0$ such that
$$
\tilde d\Big(\frac{1}{t}\int_{0}^{t}\delta_{X_{s}(y)}ds \;,\; \frac{1}{t}\int_{0}^{t}\delta_{X_{s}(x)}ds \Big) < \vep
	\quad\text{for every $y \in B(x , t, \delta)$ and $t\ge 0$.}
$$
%for every $y \in B(x , t, \delta)$ and $t\ge 0$.
\end{lemma}

\begin{proof}
Since the map $M \ni x \mapsto \delta_{x}$ is uniformly continuous, for any $\vep > 0$ there exists
$\delta > 0$ so that if $d(x, y) < \delta$ then $\tilde d(\delta_{x},\delta_{y}) < \vep$. Hence, if %$x \in Y_{ c, n}$ and
$y \in B(x,n,\vep)$ we have
$
\tilde d\Big(\frac{1}{t}\int_{0}^{t}\delta_{X_{s}(y)}ds \;,\; \frac{1}{t}\int_{0}^{t}\delta_{X_{s}(x)} \; ds \Big)
	\leq \frac{1}{t}\int_{0}^{t} \tilde d(\delta_{X_{s}(x)} \,,\, \delta_{X_{s}(y)} ) \; ds
	< \vep.
$
\end{proof}

The remaining of this section is devoted to two results on distance and entropy approximation of invariant measures
by ergodic ones. Recall  %To develop a thermodynamical formalism for flows, notions like entropy need to be defined in a %significant way. It is somewhat natural to define
the entropy of an invariant measure $\mu$ for the flow $(X_t)_{t\in \mathbb R} $
as the entropy $h_\mu(X_1)$ of the time-$1$ map (see e.g. \cite{BR75}).
The first result is a consequence of the ergodic decomposition theorem,
whose proof can be found e.g. in \cite[Lemma 2.11]{AB11}. %\cite[Lemma 4.3]{Va12} and

\begin{lemma}\label{lemAB}
Let $f: M \to M$ be a continuous map on a metric space $M$.
Let $\eta$ be an $f$-invariant probability measure  and $\psi , \phi : M \rightarrow \R$ be functions in $L^{1}(\nu)$.
Given $\vep > 0$ there exists
$\eta_{1}, \ldots ,\eta_{n}$ $f$-invariant and ergodic probabilities and $a_1, \dots, a_n >0$,
with $\sum_{i = 1}^{n}a_{i} = 1$, such that
(i) $|\int\psi d\eta - \int\psi \; d (\sum_{i=1}^{n}a_{i}\eta_{i}) | < \vep$;
(ii) $|\int\phi d\eta - \int\phi \; d (\sum_{i=1}^{n}a_{i}\eta_{i})| < \vep$
and
(iii) $h_{\eta} (f) \le  \sum_{i=1}^{n} \, a_{i} h_{\eta_{i}}(f)) + \vep$
\end{lemma}

A more general approximation result, from which the later follows immediately and that considers
the weak$^*$ topology, is as follows:

\begin{lemma}\label{lemAB2}
Let $f: M \to M$ be a continuous map on a compact metric space $M$.
Let $\eta$ be an $f$-invariant probability measure and $\tilde d$ be the usual metric in the weak$^{\star}$-topology.
Given $\vep > 0$ there exists $\eta_{1}, \ldots ,\eta_{n}$ $f$-invariant and ergodic probabilities
and $a_1, \dots, a_n >0$, with $\sum_{i = 1}^{n}a_{i} = 1$, such that (i)
$\tilde d(\eta , \sum_{i=1}^{n}a_{i}\eta_{i} ) < \vep$
and  (ii) $h_{\eta} (f) \le  \sum_{i=1}^{n} \, a_{i} h_{\eta_{i}}(f)) + \vep$.
\end{lemma}

\begin{proof}
Let $\eta$ be an $f$-invariant probability measure.
By ergodic decomposition theorem and convexity of the
entropy function (see e.g.~\cite{Wal}), we can write $ \eta = \int \eta_x \,d\eta(x) $ and
$h_\eta(f) = \int h_{\eta_x}(f) \,d\eta(x),$ where each $\eta_x$
denotes an ergodic component of $\eta$.
Take a small finite partition $\cP$ of
the space $\cM(M)$ of invariant probability measures
with diameter smaller than $\vep$.
%supported in
%$\La$ such that
%\begin{equation}\label{eq.approx}
%\Big| \int \psi_j \,d\xi_1-\int \psi_j \,d\xi_2 \Big| <\beta
%\end{equation}
%for every $1\le j\le e$ and every pair of probability measures
%$\xi_1, \xi_2$ in the same partition element.
Set $n=\#\cP$ and $a_i=\eta(\{x\in M : \eta_x \in P_i\})$ for every element $P_i$ in $\cP$. For every $1\leq i
\leq n$ pick an ergodic measure $\eta_i=\eta_{x_i} \in P_i$
satisfying $h_{\eta_x}(f) \leq h_{\eta_i}(f)+\vep$ for
$\eta$-almost every $\eta_x\in P_i$.
Part (i) in the lemma is immediate. On the other hand, (ii) follows
because
$$
h_\eta(f)
        = \int h_{\eta_x}(f) \,d\eta(x) \leq \sum_{i=1}^n a_i \,h_{\eta_i}(f)+\vep
        = h_{\hat\eta}(f)+\vep.
$$
Finally, by convexity of the metric $\tilde d$ we get
$$
\tilde d \Big( \int \eta_x  \;d\eta(x),  \sum_{i=1}^n a_i \eta_i \Big)
	= \tilde d \Big( \sum_{i=1}^n  \int_{P_i} \eta_x  \;d\eta(x),  \sum_{i=1}^n a_i \eta_i \Big)
        \leq \vep.
$$
This finishes the proof of the lemma.
\end{proof}

%%%%%%%%%%%%%%%%%%%
\subsection{Proof of the Theorem \ref{thm:LB}}
We prove the upper and lower bounds separately. We will need to recall some necessary notions.
Given $t, \vep>0$ we say that a set $E \subset M$ is a \emph{$(t,\vep)$-separated set} for the flow if
$\max_{s\in [0,t]} d(X_s(x), X_s(y)) >\vep$ for any $x\neq y\in E$. We say that $E$ is a maximal
$(t,\vep)$-separated set if it is a separated set with maximal cardinality (exist by compactness of $M$).
Similarly, given $n\in\mathbb N$ and $\vep>0$, a set $E\subset M$ is  \emph{$(n,\vep)$-separated}
if $\max_{0 \le j \le n} d(X_j(x), X_j(y)) >\vep$ for any $x\neq y\in E$.

%%%%%%%%%%%%%%%%%%%%%%%%%%%%%%%%%%%%%%%%%%%%%%
\subsubsection{Upper bound}

\noindent The proof of the upper bound combines the method for estimating large deviations for the time-$1$ map $X_1$, potential $\phi_1=\int_0^1 \phi\circ X_s\, ds$ and observable $\psi_1=\int_0^1 \psi\circ X_s \, ds$, with an argument to construct
flow invariant measures with pressure at least as large as the pressure of any given $X_1$-invariant probability measure.
Given $T>0$ let $B_T$
denote the set of points $x\in M$ so that $\frac1T \int_0^T \psi(X_s(x)) \in [a,b]$. We observe that %\footnote{falta}
$$
\limsup_{T\to\infty} \frac1T \log \mu(B_T)
	= \limsup_{n \to\infty} \frac1n \log \mu(B_n)
$$
and $B_n$ is the set of points $x\in M$ for which $\frac1n S_n \psi_1(x) \in [a,b]$, where
$S_n\psi_1= \sum_{j=0}^{n-1} \psi_1\circ X_j$.
%Since $\mu$ is a weak Gibbs measure then
%$$
%\frac{1}{K_{t}(\vep)} \leq \frac{\mu(B(x, t, \vep))}{\exp\Big[\int_{0}^{t}\phi(X_{s}(x))ds - tP_{\mu}\Big]} \leq K_{t}(\vep),
%$$
%for every $t>0$,
%where $P_{\mu} := h_{\mu} (X_1)+ \int\phi \, d\mu$
%and $\limsup_{t\to 0} \frac1t \log K_{t}(\vep) =0$.
If $E_n\subset B_n$ is a maximal $(n,\vep)$-separated set for the flow then $B_n \subset \bigcup_{x\in E_n} B(x,n,2\vep)$
and it follows from the Gibbs property ~\eqref{eq:Gibbss} that
\begin{equation*}\label{eq.upper}
\mu(B_n)
	\le  K_n(\vep) e^{-n P_\mu} \, \sum_{x \in E_n} e^{\int_0^n \phi(X_s(x)) \, ds}
	= K_n(\vep) e^{-n P_\mu} \, \sum_{x \in E_n}  e^{S_n \phi_1(x)}
\end{equation*}
for every $n\ge 1$. Thus
\begin{align*}
\limsup_{n \to\infty} \frac1n \log \mu(B_n)
	\le  - P_\mu
	+ \limsup_{n\to\infty} \frac1n \log Z_n,
\end{align*}
where $Z_n=\sum_{x \in E_n} e^{S_n\phi_1(x)}$. Now, given $\vep>0$, %is chosen small,
by uniform continuity of $X_t$ for $t\in [0,1]$ there exists
%If $E_T\subset B_T$ is a maximal $(T,\vep)$-separated set for the flow then, $B_T \subset \bigcup_{x\in E_T} B(x,T,2\vep)$
%and, consequently,
%\begin{equation*}\label{eq.upper}
%\mu(B_T) \le  K_T(\vep) e^{-T P_\mu} \, \sum_{x \in E_T} e^{\int_0^T \phi(X_s(x)) \, ds}
%\end{equation*}
%for every $T>0$. This implies that
%\begin{align*}
%\limsup_{T\to\infty} \frac1T \log \mu(B_T)
%	= \limsup_{n \to\infty} \frac1n \log \mu(B_n)
%	\le  - P_\mu
%	+ \limsup_{n\to\infty} \frac1n \log Z_n,
%\end{align*}
%where $Z_T=\sum_{x \in E_T} e^{\int_0^T \phi(X_s(x)) \, ds}$. Moreover, if $\phi_1(x) = \int_{0}^t \phi(X_s(x)) \, ds$ and $\vep>0$
%is chosen small (by uniform continuity of $X_t$ for $t\in [0,1]$)
$\zeta \in (0,1)$ so that any $(n,\vep)$-separated set for the flow is $(n,\zeta \vep)$-separated set for
the time one map $X_1$. Thus $ E_n$ is a $(n,\zeta \vep)$-separated set for
the time one map $X_1$.
%Following \cite[page~185]{BR75})
%the topological pressure for continuous flows on compact metric spaces coincides with
%the topological pressure for the time-$1$ map with respect to the induced potential $\phi_1$, that is,
%$
%\Ptop( (X_t)_t, \phi)
%	= \Ptop( X_1, \, \phi_1).
%$
%
Following \cite{You90}, consider the probability measures $\si_n$ and $\eta_n$ given by
$$
\si_n
   =\frac{1}{Z_n} \sum_{x \in  E_n} e^{S_n \phi_1 (x)} \, \delta_x
   \quad\text{and}\quad
\eta_n
    = \frac{1}{n}  \sum_{j=0}^{n-1} (X_j)_* \si_n.
$$
Clearly, any  weak$^*$ accumulation point $\eta$ of the sequence $(\eta_n)_{n\in\mathbb N}$ is
an $X_1$-invariant probability measure. Let $\cP$ be a partition of $M$ with diameter smaller
than $\zeta \vep$ and $\eta(\partial\cP)=0$. By construction every element of $\cP^{(n)}=\bigvee_{j=0}^{n-1} X_{-j}(\mathcal P)$
contains at most one point of $E_n$.
Thus
$$
H_{\si_n}(\cP^{(n)}) - \int S_n \phi_1 \,d\si_n
    = \log \Big( \sum_{x\in E_n} e^{S_n \phi_1} \Big)
$$
which, as in the usual proof of the variational principle (c.f.\cite[Pages 219-221]{Wal}), guarantees that
\begin{equation*}\label{eq.upper2}
\limsup_{n \to \infty}
	    \frac{1}{n} \log  \sum_{x \in  E_n} e^{ \sum_{j=0}^{n-1} \phi_1 (X_j(x))}
    \leq h_\eta(f) + \int \phi_1 \, d\eta.
\end{equation*}
 Observe also that $\int \psi_1 \,d\eta \in  [a,b]$ by weak$^*$ convergence, because $E_n$ is contained in $B_n$ and
$$
 \int \psi_1 \,d\eta_n
        = \frac{1}{n}  \sum_{j=0}^{n-1} \frac{1}{ Z_n} \sum_{x \in E_n} e^{S_n \phi_1(x)} \psi_1(X_j(x))
        \in [a,b].
$$
%Following Walters (see also e.g. \cite[page~185]{BR75}))
The probability measure $\tilde \eta:= \int_0^1 (X_s)_* \eta \, ds$ is clearly flow
invariant and each probability measure $(X_s)_* \eta$ is $X_1$-invariant with the same entropy as $\eta$. Thus
$$
h_{\tilde \eta} (X_1)
	= \int_0^1 h_{(X_s)_* \eta} (X_1) \, ds
	= h_{\eta} (X_1)
	\quad\text{and}\quad
\int \phi \, d\tilde \eta
	= \int \int_0^1 \phi\circ X_s \, ds \, d \eta
	= \int \phi_1 \, d \eta.
$$
This yields that
$
h_\eta(f) + \int \phi_1 \, d\eta \le h_{\tilde \eta}(f) + \int \phi_1 \, d\tilde \eta.
$
Since
$
\int \psi \,d \tilde \eta = \int \int_0^1 \psi (X_s(x)) \,d \eta = \int \psi_1 (x)) \,d \eta  \in  [a,b],
$
this finishes the proof of the first part of the theorem.

%%%%%%%%%%%%%%%%%%%%%%%%%%%%%%%%%%%%%%%%%%%%%%
\subsubsection{Lower bound}\label{1lb}

Set $f=X_1$ as the time-1 map of the flow $(X_t)_{t\in\mathbb R}$, $(a,b)\subset \mathbb R$ be an
open interval and $\phi,\psi$ be bounded observables with tempered variation. Given $t>0$ consider the set
$$
D_{t}  := \Big\{x \in M : \frac{1}{t}\int_{0}^{t}\psi(X_{s}(x)) ds \in (a , b)\Big\}.
$$
Given any $f$-invariant probability measure $\nu$ satisfying $\int \psi d\nu \in (a , b)$ and $\vep > 0$
we claim that there exists $t_1>0$ so that
$$
\mu(D_{t}) \geq \exp t\Big[h_{\nu}(X_{1}) + \int \phi d\nu - P_{\mu} - \vep]
$$
for every $t \ge t_1$. Since this claim, together with
\begin{align*}
 \inf\Big\{
 	& P_\mu - h_{\nu}(X_1) - \int\phi \, d\nu  : \nu \text{ is } X_{1}\text{-invariant and }
	 \int \psi \, d\nu \in (a , b) \Big\} \\
	 & \leq
	 \inf\Big\{P_\mu - h_{\nu}(X_1) - \int\phi \, d\nu  : \nu \text{ is } (X_{t})_t\text{-invariant and }
	 \int \psi \, d\nu \in (a , b) \Big\},
\end{align*}
implies the statement of the theorem we are left to prove it.

Fix $\nu$ as above and $\vep>0$ arbitrary. Take $\vep_0 := \frac{1}{6} \min\{\vep , |a - \int\psi d\nu| , |b - \int\psi d\nu|\}$
and let $\nu_{1}, \ldots, \nu_{n}$ be $X_{1}$-invariant and ergodic probability measures
so that $\hat \nu = \sum_{i=1}^{n}a_{i}\nu_{i}$ is $\vep_0$-approximating $\nu$
in the sense of (i)-(iii) in Lemma \ref{lemAB}.
For any $i=1\dots n$ consider the sets
\begin{align*}
E_{t}^{i}:= \Big\{x \in M :
	& \Big|\frac{1}{\lfloor a_{i}t\rfloor}\int_{0}^{\lfloor a_{i}t\rfloor}\psi(X_{s}(x))ds -
	\int\psi d\nu_{i}\Big| < \vep_0 \\
	& \&  \Big|\frac{1}{\lfloor a_{i}t\rfloor}\int_{0}^{\lfloor a_{i}t\rfloor}\phi(X_{s}(x))ds - \int\phi d\nu_{i}\Big| < \vep_0\Big\},
\end{align*}
where $\lfloor a_{i}t\rfloor$ denotes the integer part of $a_i t$.

Using Birkhoff's ergodic theorem and that entropy can be computed via separated sets
(c.f. \cite{Katok,Wal}), there are $t_{1} > 0$,  $0 < \vep_{1},\vep_{2} \leq \vep_0$ and a maximal
$(\lfloor a_{i}t\rfloor , \vep_{2})$-separated set $N_{t}^{i} =\{x_{t,1}^{i}, \ldots, x_{t,m_{t}^{i}}^{i}\}$
of cardinality $m_{t}^{i} \geq \exp\big[ \lfloor a_{i}t\rfloor(h_{\nu_{i}}(f) - \vep_{1})\big]$ for every $t\ge t_1$
and $i=1\dots n$.
Up to increase $t_1$ if necessary, Lemma \ref{lem1} guarantees that there exists $0< \delta \leq \vep_{2}$ small so
that $B(x , t , \delta) \subset D_{t}$ for every  $x \in D_{t}$ and $t>t_1$.

By construction $N_{t}^{i}$ is a $(\lfloor a_{i}t\rfloor , \delta)$-separated set.
We now make use of the gluing orbit property at scale $\frac\delta4$. Indeed,  for any $1 \leq j_{i} \leq m_{t}^{i}$, with $i = 1, \ldots,n$, by the gluing orbit property
one can pick  $y \in M$ that shadows the pieces of orbits of the points $x^1_{t, j_1}, x^2_{t, j_2}, \dots, x^n_{t, j_n}$,
for $ 1\le j_i \le m_{t}^{i}$ within a distance $\frac\delta4$, by times $\lfloor
a_{i}t\rfloor$ and with jump times  $p_{1}, \ldots, p_{n-1} \leq T(\frac\delta4)$ between each shadowing segment.
Let $Y_{t}$ be the set of all such choices of points $y$.
%By uniform continuity of both $\psi,\phi$ we may assume that $\delta>0$ is small so
%that $|\psi(x)-\psi(y)|, |\phi(x)-\phi(y)|<\vep_0$ for all $d(x,y)<\delta$.
Since $\psi$ has tempered variation we may assume $\delta>0$ is small so that $C_\delta(\psi)<\vep_0/6$
(recall the definition in equation~\eqref{eq:distort}). %We assume also $\psi,\phi$ bounded.

\begin{lemma}\label{lem5}
If $t_1>0$ is large then $Y_{t} \subset D_{t + nT(\frac\delta4)}$ for every $t \geq t_{1}$.
\end{lemma}

\begin{proof}
Take $y \in Y_{t}$ and let $x^1_{t, j_1}, x^2_{t, j_2}, \dots, x^n_{t, j_n}$ be the points that determined the choice of $y$.
Splitting the pieces of the orbit of $y$ up to time $t+n T(\frac\delta4)$ according to its shadowing paths of size $\lfloor a_{i}t\rfloor$ and their complements,  and setting $p_0=a_0=0$, then
\begin{align*}
\int_{0}^{t+nT(\frac\delta4))}
	\psi(X_{s}(y)) ds
	&
=
\sum_{i=1}^{n}\int_{0}^{\lfloor a_{i}t\rfloor}\psi(X_{s +  \sum_{j = 0}^{i-1} (\lfloor a_{j}t\rfloor + p_{j}) }(y)) \, ds \\
& +
\sum_{i = 1}^{n}
	\int_{\lfloor a_{i}t\rfloor + \sum_{j = 1}^{i-1} (\lfloor a_{j}t\rfloor + p_{j}) }^{\sum_{j = 1}^{i} (\lfloor a_{j}t\rfloor + p_{j})}
	\psi(X_{s  }(y))ds \nonumber\\
 &  + \int_{\sum_{j = 1}^{n} ( \lfloor a_{j}t\rfloor + p_{j} ) }^{t + n T(\frac\delta4))}\psi(X_{s  }(y))ds \nonumber
 \end{align*}
 where the first term in the right hand sum differs from $\sum_{i=1}^{n}\int_{0}^{\lfloor a_{i}t\rfloor}\psi(X_{s}(x_{t,j_{i}}^{i}))ds$
 by at most $\frac{\vep_0}{6} (t + nT(\frac\delta4)))$ by the tempered variation property of $\psi$ and choice of $\delta$.
Using that $p_{1}, \ldots, p_{n-1} \leq T(\frac\delta4))$ (with $T(\frac\delta4))$ independent of $t$) up to consider a larger $t_{1}>0$
the sum of the two last summands in the right hand side is bounded above by $2 \|\psi\|_{L^\infty} n T(\frac\delta4))$.
Finally, using $|\int\psi d\nu - \int\psi d (\sum_{i=1}^{n}a_{i}\nu_{i})| < \vep_0$, $x_{t,j_{i}}^{i} \in E_{t}^{i}$, a simple computation
using the tempered variation condition for $\psi$ yields that
 %\begin{align}
%\int_{0}^{t+nT(\delta)}
%	\psi(X_{s}(y)) ds
%	&
%=
%\sum_{i=1}^{n}\int_{0}^{\lfloor a_{i}t\rfloor}\psi(X_{s +  \sum_{j = 0}^{i-1} (\lfloor a_{j}t\rfloor + p_{j}) }(y)) \, ds
%
%+
%\sum_{i = 1}^{n}
%	\int_{\lfloor a_{i}t\rfloor + \sum_{j = 1}^{i-1} (\lfloor a_{j}t\rfloor + p_{j}) }^{\sum_{j = 1}^{i} (\lfloor a_{j}t\rfloor + p_{j})}
%	\psi(X_{s  }(y))ds \nonumber\\
%
% &  + \int_{\sum_{j = 1}^{n} ( \lfloor a_{j}t\rfloor + p_{j} ) }^{t + n T(\delta)}\psi(X_{s  }(y))ds \nonumber \\
 %
  %          & = \sum_{i=1}^{n}\int_{0}^{\lfloor a_{i}t\rfloor}\psi(X_{s
%}(x_{t,j_{i}}^{i}))ds  \label{equa1}\\
  %       & + \sum_{i=1}^{n}\int_{0}^{\lfloor a_{i}t\rfloor}\psi(X_{s + \lfloor
%a_{i}t\rfloor + p_{i-1}}(y)) - \psi(X_{s }(x_{t,j_{i}}^{i}))ds \label{equa2} \\
   % & + \sum_{i = 1}^{n}\int_{\sum_{j = =1}^{i}\lfloor a_{i}t\rfloor +
%\sum_{j= 1}^{i-1}p_{j} }^{\sum_{j = =1}^{i}\lfloor a_{i}t\rfloor +
%\sum_{j= 1}^{i}p_{j}}\psi(X_{s  }(y))ds + \sum_{i =
%1}^{n}\int_{\sum_{j = =1}^{i}\lfloor a_{i}t\rfloor + \sum_{j=
%1}^{i}p_{j} }^{t + n T(\delta)}\psi(X_{s  }(y))ds. \label{equa3}
 %\end{align}
\begin{align*}
\frac{1}{\lfloor t+nT(\frac\delta4)) \rfloor}\int_{0}^{t+nT(\frac\delta4))}
	\psi(X_{s}(y)) ds
	\in \big( \int \psi  \, d\nu -\vep_0,  \int \psi  \, d\nu+\vep_0 \big)
	\subset (a, b)
\end{align*}
for every $t \ge t_1$, proving the lemma.
\end{proof}

 We claim that there exists $C>0$ (depending on the vector field $X$, $n$ and $\delta$)
and a subset $\tilde Y_t \subset Y_t$ with cardinality larger or equal to
$
C\cdot  \prod_{i=1}^n \exp\big[ \lfloor a_{i}t\rfloor(h_{\nu_{i}}(f) - \vep_{1})\big]
$
such that the family of dynamical balls $\{B(y , t +nT(\frac\delta4), \frac{\delta}{4}) \}_{y \in \tilde Y_{t}}$ is a disjoint family of
%By construction,  $\{B(y , t +nT(\delta), \frac{\delta}{2}) \}_{y \in Y_{t}}$ is a disjoint family of
subsets of $D_{t + nT(\frac\delta4)}$ for every $t \geq t_{1}$.
Recall that each $y\in Y_t$ is determined by points $x^1_{t, j_1}, x^2_{t, j_2}, \dots, x^n_{t, j_n}$ ($1 \leq j_{i} \leq m_{t}^{i}$),
%with $i = 1, \ldots,n$
by the shadowing times $\lfloor a_{i}t\rfloor$ and by gluing times $0\le p_{1}, \ldots, p_{n-1} \leq T(\frac\delta4)$
between each shadowing segment (we also write $p_i=p_i(y)$ to emphasize these are functions of the underlying points
and times).
Let $t(\delta,n)>0$ be %given by uniform continuity
so that $\max_{s\in [0,n \, t(\delta,n)]} d_{C^0}(X_s,Id) <\frac{\delta}4$ and write
\begin{equation}\label{eq:decompos}
[0,T(\frac\delta4)] = \bigcup_{s=0}^{N(\delta,n)-1} I_s
	\; \bigcup \; I_{N(\delta,n)}
\end{equation}
where $N(\delta,n)=\big[\frac{T(\frac\delta4)}{t(\delta,n)}\big]$ denotes the integer part of $\frac{T(\frac\delta4)}{t(\delta,n)}$,
$I_s = [s t(\delta,n), (s+1) t(\delta,n)[$ for $0\le s \le N(\delta,n)-1$, and the last interval $I_{N(\delta,n)}$
%= [T(\frac\delta4) - N(\delta,n) t(\vep), T(\frac\delta4)] $
has also size bounded by $t(\delta,n)$ and may be reduced to a point.
By the pigeonhole principle (recall \eqref{eq:decompos}), for every $1\le i \le n-1$
there exists $0\le s_i \le N(\delta,n)$
so that the set
\begin{equation}\label{eq:estYt}
\tilde Y_t := \{ y \in Y_t \colon \; p_i(y) \in I_{s_i} \; \text{for every }\; 1\le i \le n-1 \}
\end{equation}
has cardinality at least $ C \cdot \# Y_t$ where $C=\frac{1}{N(\delta,n)^{n-1}}$.
%We claim that the set $\Big\{ z_{p} \in B(x,\frac\vep{4}) \colon \text{ period of } p \text{ belongs to } I_{j}\Big\}$ is
%$((1+\vep) (T_k + 2K(\frac\vep4)), \delta_{0} - \vep)-$separated in $U$.
The family $\{B(y , t +nT(\frac\delta4), \frac{\delta}{4}) \}_{y \in \tilde Y_{t}}$ is disjoint. Indeed,
if $y_1\neq y_2 \in \tilde Y_t$, there are $1\le i \le n$ and points $x_1\neq x_2 \in N^i_t$
so that
$$
d(X_{\ell + \sum_{j=0}^{i-1} \lfloor a_j t\rfloor + p_j(y_1) }(y_1), X_\ell (x_1)) < \frac\delta4
	\quad\text{and}\quad
d(X_{\ell + \sum_{j=0}^{i-1} \lfloor a_j t\rfloor + p_j(y_2)}(y_2), X_\ell (x_2)) < \frac\delta4
$$
for $\ell \in [0, \lfloor a_i t\rfloor]$. Moreover, as points in $N^i_t$ are $(\lfloor a_{i}t\rfloor,\delta)$-separated,
%using that $x_1\neq x_2$
there exists $s_i\in [0,\lfloor a_{i}t\rfloor]$ such that
$
d(X_{s_i}(x_1), X_{s_i}(x_2)) >\delta.
$
Thus, by triangular inequality and the fact that $\sum_{j=0}^{i-1} |p_j(y_1)-p_j(y_2)|<n \, t(\delta,n)$, we conclude that
$
d(X_{\tilde s_i}(y_1),X_{\tilde s_i}(y_2)
	> \frac\delta4
$
where $0\le \tilde s_i := s_i + \sum_{j=0}^{i-1} \lfloor a_j t\rfloor + p_j(y_1) \le t+ nT(\frac\delta4)$. This proves the claim.
%$$
%d(X_{s_i + \sum_{j=0}^{i-1} \lfloor a_j t\rfloor + p_j(y_1) }(y_1),
%	X_{s_i + \sum_{j=0}^{i-1} \lfloor a_j t\rfloor + p_j(y_1)}(y_2))
%	> \frac\delta4
%$$

Finally, we estimate $\mu(D_{t + n T(\frac\delta4)})$, for every $t \geq
t_{1}$. Estimates similar to the ones of the previous lemma yield that
$
\frac{1}{{t + nT(\frac\delta4)}} \int_{0}^{t + nT(\frac\delta4)}\phi(X_{s}(y)) ds
\in \Big( \int \phi \,d\nu -\vep_0  ,  \int \phi \,d\nu + \vep_0 \Big)
$
for all $ t \geq t_{1}$. Thus,
\begin{align}
%\mu(D_{t})
\mu ( D_{t + n T(\frac\delta4)}  )
	& \geq \sum_{y \in \tilde Y_{t}}\mu\Big( B(y , t + nT(\frac\delta4), \frac{\delta}{2}) \Big) \\
	& \geq \frac{1}{K_{t + n T(\frac\delta4)}(\frac{\delta}{2})} \sum_{y \in \tilde Y_{t}}
	\exp \Big[ \int_{0}^{t + nT(\frac\delta4)} \phi(X_{s}(y))ds - (t + nT(\frac\delta4)P_{\mu} \Big]  \nonumber \\
       & \geq \frac{1}{K_{t + n T(\frac\delta4)}(\frac{\delta}{2})} \# \tilde Y_{t} \cdot \exp
\Big[ (t + nT(\frac\delta4)) (\int\phi d\nu - \vep_0) - (t +
nT(\frac\delta4)P_{\mu} \Big] \nonumber\\
        & \geq \frac{C}{K_{t + n T(\frac\delta4)}(\frac{\delta}{2})} \exp\Big[ \sum_{i = 1}^{n}\lfloor a_{i}t\rfloor(h_{\nu_{i}} - \vep_0) + (t + nT(\frac\delta4)) (\int\phi d\nu - \vep_0) - (t + nT(\frac\delta4))P_{\mu} \Big]. \nonumber
\end{align}
Since $|h_{\nu}(f) - h_{\sum_{i = 1}^{n}a_{i}\nu_{i}}(f) |< \vep_0$ and
$\lim_{t \to \infty}\frac{1}{t}\log K_{t}(\frac{\delta}{2}) =  0$, one can take $t_1>0$ large so that the claim holds.
This completes the proof of the theorem.

%%%%%%%%%%%%%%%%%%%%%%%%%%%%%%%%%%%%%%%%%%%%%%%%%%%%%%%%%%%%%%%
\subsection{Proof of Theorem~\ref{thm:LB2}}

Since this proof has similar ingredients to the one of Subsection~\ref{1lb} we shall concentrate on the main differences.
Fix an open set $V$ in the space of all probabilities in $M$ and, for $t>0$, consider the set
$$
D_{t}  := \Big\{x \in M : \frac{1}{t}\int_{0}^{t}\delta_{X_{s}(x)} ds \in V\Big\}.
$$
Take a $X_{1}$-invariant probability measure $\nu \in V$ and $\vep > 0$.
We claim that there exists $t_1>0$ so that
$$
\mu(D_{t}) \geq \exp t\Big[h_{\nu}(X_{1}) + \int\phi d\nu - P_{\mu} - \vep\Big]
$$
for every $t \ge t_1$, which will imply the theorem.
Fix $\nu$ as above and $\vep>0$ arbitrary.

Take $0<\hat{\vep} \leq \vep$ be such that $B_{\tilde d}(\nu, \hat{\vep}) \subset V$ (the ball is taken with respect to the
metric $\tilde d$) and set $\vep_0 := \frac{\hat{\vep} }{6}$.
Let $\nu_{1}, \ldots, \nu_{n}$ be the $X_{1}$-invariant and ergodic probability measures so that $\hat \nu = \sum_{i=1}^{n}a_{i}\nu_{i}$ is $\vep_0$-approximating $\nu$ in the sense of (i)-(ii) in Lemma \ref{lemAB2}.  By Birkhoff's ergodic theorem there
exists $t_1>0$ so that the sets
%
%For any $i=1,\dots, n$ consider the sets
%\begin{align*}
%E_{t}^{i}:= \Big\{x \in M :
%	& \tilde d \big( \frac{1}{\lfloor a_{i}t\rfloor}\int_{0}^{\lfloor a_{i}t\rfloor}\delta_{X_{s}(x)}ds, \nu_i \big) < \vep_0 \\
%	& \&  \Big|\frac{1}{\lfloor a_{i}t\rfloor}\int_{0}^{\lfloor a_{i}t\rfloor}\phi(X_{s}(x))ds - \int\phi d\nu_{i}\Big| < \vep_0\Big\},
%\end{align*}
\begin{align*}
E_{t}^{i}:= \Big\{x \in M :
	& \; \tilde d \Big( \frac{1}{\lfloor a_{i}t\rfloor}\int_{0}^{\lfloor a_{i}t\rfloor} \delta_{X_{s}(x)} ds, \nu_{i} \Big) < \vep_0 \\
	& \&  \Big|\frac{1}{\lfloor a_{i}t\rfloor}\int_{0}^{\lfloor a_{i}t\rfloor}\phi(X_{s}(x))ds - \int\phi \, d\nu_{i}\Big| < \vep_0\Big\},
\end{align*}
(where $\lfloor a_{i}t\rfloor$ denotes the integer part of $a_i t$)  satisfy $\nu_i(E_{t}^{i}) \ge \frac12$ for every
$t\ge t_1$ and $i=1\dots n$.
Again, since entropy can be computed via separated sets (c.f. \cite{Katok}), up to increase $t_{1} > 0$ if necessary, there
are  $0 < \vep_{1},\vep_{2} \leq \vep_0$ and a maximal
$(\lfloor a_{i}t\rfloor , \vep_{2})$-separated set $N_{t}^{i} =\{x_{t,1}^{i}, \ldots, x_{t,m_{t}^{i}}^{i}\}$
of cardinality $m_{t}^{i} \geq \exp\big[ \lfloor a_{i}t\rfloor(h_{\nu_{i}}(f) - \vep_{1})\big]$ for every $t\ge t_1$
and $i=1\dots n$.
Lemma \ref{lem2} guarantees that there exists $0< \delta \leq \vep_{2}$ so that
$$
\tilde d\Big(\frac{1}{t}\int_{0}^{t}\delta_{X_{s}(y)}ds \;,\; \frac{1}{t}\int_{0}^{t}\delta_{X_{s}(x)}ds \Big) < \frac{\vep_0}{6}
$$
for every $y \in B(x , t, \delta)$. Up to increase $t_1$ if necessary, %Lemma \ref{lem1} and
%Lemma \ref{lem2} guarantees  that
this implies that there exists $0< \delta \leq \vep_{2}$ small so that $B(x , t , \delta) \subset D_{t}$ for
every  $x \in D_{t}$ and $t>t_1$.

We now make use of the gluing orbit property
for the scale $\delta>0$. Indeed,  for any $1 \leq j_{i} \leq m_{t}^{i}$, with $i = 1, \ldots,n$, by the gluing orbit property
one can pick  $y \in M$ that shadows the pieces of orbits of the points $x^1_{t, j_1}, x^2_{t, j_2}, \dots, x^n_{t, j_n}$,
for $ 1\le j_i \le m_{t}^{i}$ within a distance $\delta$, by times $\lfloor
a_{i}t\rfloor$, respectively, and with jump times  $p_{1}, \ldots, p_{n-1} \leq T(\delta)$ between each shadowing segment.
Denote the set of all such choices of points $y$ as the set $Y_{t}$. Since $\tilde d(\cdot, \nu)$ is a convex function then
the same ideas as in Lemma~\ref{lem5} are enough to prove that $Y_{t} \subset D_{t + nT(\delta)}$ for every $t \geq t_{1}$.

Observe the sets $N_{t}^{i}$ are $(\lfloor a_{i}t\rfloor , \vep_{2})$-separated and
$0<\delta<\vep_2$.  The same argument used in the proof of Theorem~\ref{thm:LB} implies that
there exists $C>0$ (depending on the vector field $X$, $n$ and $\delta$)
and a subset $\tilde Y_t \subset Y_t$ with cardinality larger or equal to
$
C\cdot  \prod_{i=1}^n \exp\big[ \lfloor a_{i}t\rfloor(h_{\nu_{i}}(f) - \vep_{1})\big]
$
such that the family of dynamical balls $\{B(y , t +nT(\frac\delta4), \frac{\delta}{4}) \}_{y \in \tilde Y_{t}}$ is a disjoint family of
subsets of $D_{t + nT(\frac\delta4)}$ for every $t \geq t_{1}$
 and
one can estimate $\mu(D_{t + n T(\frac\delta4)})$ similarly to the proof of Lemma \ref{lem5}
(using that $\phi$ has tempered variation and it is bounded):
%$$
%\frac{1}{{t + nT(\delta)}} \int_{0}^{t + nT(\delta)}\phi(X_{s}(y)) ds
%\in \Big( \int \phi d\nu -\vep_0  ,  \int \phi d\nu + \vep_0 \Big)
%$$
%for all $ t \geq t_{1}$. Thus, one can estimate the measure of $D_{t + n T(\delta)}$ by
\begin{align}
\mu(D_{t+nT(\frac\delta4)})
	& \geq \sum_{y \in \tilde{Y}_{t}}\mu\Big( B(y , t + nT(\frac\delta4), \frac{\delta}{2}) \Big) \\
	& \geq \frac{1}{K_{t + n T(\frac\delta4)}(\frac{\delta}{2})} \sum _{y \in \tilde{Y}_{t}}
	\exp \Big[ \int_{0}^{t + nT(\frac\delta4)} \phi(X_{s}(y))ds - (t + nT(\frac\delta4))P_{\mu} \Big]  \nonumber \\
       & \geq \frac{1}{K_{t + n T(\frac\delta4)}(\frac{\delta}{2})} \#\tilde{Y}_{t} \cdot \exp
\Big[ (t + nT(\frac\delta4)) (\int\phi d\nu - \vep_0) - (t +nT(\frac\delta4))P_{\mu} \Big] \nonumber\\
        & \geq \frac{C}{K_{t + n T(\frac\delta4)}(\frac{\delta}{2})} \exp\Big[ \sum_{i = 1}^{n}\lfloor a_{i}t\rfloor(h_{\nu_{i}} - \vep_0) + (t + nT(\frac\delta4)) (\int\phi d\nu - \vep_0) - (t + nT(\frac\delta4))P_{\mu} \Big]. \nonumber
\end{align}
This proves the claim since $|h_{\nu}(f) - h_{\sum_{i = 1}^{n}a_{i}\nu_{i}}(f) |< \vep_0$ and
$\lim_{t \to \infty}\frac{1}{t}\log K_{t}(\frac{\delta}{2}) =  0$.

%%%%%%%%%%%%%%%%%%%
\subsection{Proof of Theorem~\ref{thm:Wad}}
Let $\sigma: \Sigma \to \Sigma$ be a subshift of finite type, $\rho :\Sigma \rightarrow \R$ be a H\"older continuous roof function
and $(X_{t})_{t}$ be the suspension flow generated by $\sigma$ and $\rho$. Assume
$\phi: \Sigma_{\rho}\to \mathbb R$ is a continuous %H\"older
potential so that $\mu_{\phi}$ is the unique equilibrium state for $(X_{t})_{t}$ with respect to $\phi$
and is a Gibbs measure. Applying the Theorem \ref{thm:gluingAA} we have that $(X_{t})_{t}$ has the gluing orbit property.
%{\color{blue}By hyperbolicity of $\sigma$ we have that all continuous observable has tempered variation.} isto nao eh verdade
So, by compactness of $\Sigma_\rho$ and continuity of the observable $\psi: \Sigma_{\rho}\to\mathbb R$ it follows from Theorem~\ref{thm:LB} the following level-1 large deviations principle
\begin{align*}
\limsup_{t\to +\infty} \frac1t \log \mu_\phi \Big( x\in \Sigma_{\rho} : \frac{1}{t}\int_{0}^{t} \psi(X_s(x)) \, ds \in [a , b] \Big)
    \le -\inf_{s \in [a , b]} I(s)
\end{align*}
and
\begin{align*}
\liminf_{t\to +\infty} \frac1t \log \mu_\phi \Big( x\in \Sigma_{\rho} : \frac{1}{t}\int_{0}^{t} \psi(X_s(x)) \, ds \in (a , b) \Big)
    \ge -\inf_{s \in (a , b)} I(s)
\end{align*}
with
$I(s) =\sup \{ P_\mu - h_{\nu}(X_1) - \int\phi \, d\nu  : \nu \text{ is } (X_{t})_t\text{-invariant and } \int \psi \, d\nu =s  \}$.
Finally, we observe that for every $(X_t)_t$-invariant probability measure $\nu$ there exists a unique
$\sigma$-invariant probability measure so that $\nu=(\eta \times \Leb)/\int \rho \, d\eta$ (see e.g. \cite{BR75}).
By the Abramov formulas we get
$$
h_\nu(X_1)= \frac{h_\eta(\sigma)}{\int\rho d\eta}
	\quad\text{and}\quad
	 \int \phi \, d\nu =\frac{\int \overline{\phi} \, d\eta}{\int\rho d\eta}
$$
and, using that $\mu$ is an equilibrium state for $\phi$, it follows that $P_\mu=\Ptop((X_{t})_{t} , \phi)$. This finishes
the proof of the theorem.

%%%%%%%%%%%%%%%%%%%
\section{Criteria for gluying orbit properties}\label{sec:criteria}

%%%%%%%%%%%%%%%%%%%
\subsection{The Bowen-Walters distance}

Before proving the criteria for suspension semiflows to satisfy gluying orbit properties we
recall the Bowen-Walters distance for the suspension semiflows.
Assume that $(M,d)$ is a metric space, $f: M \to M$ is a continuous map, $\rho : M \to \mathbb R^+_0$
is a roof function and $(X_t)_{t\ge 0}$ is the suspension semiflow over $f$ acting on the space $M_\rho$
introduced in Subsection~\ref{subsec:suspension}.
If $\rho\equiv 1$ is constant
equal to one then define a \emph{horizontal distance} for points in $M \times \{t\}$ by
$$
d_h((x,t),(y,t))
    = (1-t) d (x,y) + t d(f(x),f(y))
$$
and a define a \emph{vertical distance} for points for $(x,t)$ in the orbit of $(y,s)$ by
$$
d_v((x,t),(y,s))
    = \inf \{ |r| \colon X_r(x,t)=(y,s)\}.
$$
Then, the Bowen-Walters distance $d_1((x,t),(y,s))$ is defined as the infimum of the length of paths connecting
$(x,t)$ and $(x,s)$. For an arbitrary roof function $\rho$ the Bowen-Walters distance is defined,
for every $(x,t),(x,s) \in M$, by
$$
d_\rho((x,t),(x,s)) := d_1((x,t/\rho(x)),(y,s/\rho(y))).
$$
Although this is a very natural metric, it is also hard to explicitly compute balls and dynamical balls with respect to Bowen-Walters distance. If $f$ is invertible with both $f$ and $f^{-1}$ Lipschitz,
and the roof function $\rho$ is bounded away from zero and also Lipschitz continuous
then it follows from Barreira and Saussol ~\cite[Appendix]{BS} that there exists $K>0$ so that
\begin{equation}\label{eq:metric.equivalence}
K^{-1} d_\pi ( (x,t),(y,s) )
    \le d_\rho((x,t),(x,s))
    \le K d_\pi ( (x,t),(y,s) )
\end{equation}
for any $(x,t),(y,s)\in M$, where $d_\pi$ is the pseudo-distance
\begin{equation}\label{eq:dist}
d_\pi ( (x,t),(y,s) )
    = \min\Bigg\{
    \begin{array}{l}
    d(x,y) + |s-t|, \\
    d(f(x),y) +\rho(x)- t+s, \\
    d(x,f(y)) +\rho(y) -s +t
    \end{array}
    \Bigg\}.
\end{equation}
%Since these notions are equivalent in what follows, if not stated otherwise, we will always
%deal with the suspension flow endowed with the pseudo-distance $d_\pi$.

%%%%%%%%%%%%%%%%%%%%%%%%%%%%
\subsection{Proof of Theorem~\ref{thm:gluingAA}}\label{subsecClaim}
Assume that $f: M \to M$ satisfies the  gluing orbit property %specification property
and that the roof function $\rho: M \to \mathbb R_0^+$ is bounded
above and below.
Let $\vep>0$ be arbitrary and fixed. Take points $(x_1,s_1), (x_2,s_2), \dots, (x_k,s_k) \in M_\rho$ and times
$t_1, \dots, t_k \ge 0$ arbitrary. Given any $1\le i \le k$, let $n_i=n_i(x_i, s_i, t_i) \in \mathbb N_0$ be
determined by the equation
\begin{equation}\label{eq:turns}
\sum_{j=0}^{n_i-1} \rho(f^j(x_i))
    \le s_i + t_i
    < \sum_{j=0}^{n_i} \rho(f^j(x_i)).
\end{equation}
Using that $\rho$ is uniformly continuous and satisfies condition \eqref{eq:distort},
there exists $0<\xi<\vep/3$ be small so that $\xi  + C_\xi <\frac{\vep}{3}$,
that $C_\xi < \frac\vep{3} \inf_{x\in M} \rho(x)$
and
$|\rho(z) - \rho(w)| < \frac{\vep (\inf \rho)^2}{3 \sup \rho}$
for all $d(z,w) < \xi$.
%
%By the continuity of $r$ on the compact metric space $M$ then for any $\xi>0$ the oscillation
%\begin{equation*}
%\osc(r,\xi) := \sup \{ |r(x)-r(y)| \colon d(x,y) <\vep\}
%\end{equation*}
%is finite and $\lim_{\xi\to 0} \osc(r,\xi)=0$.
%Thus, there exists $\xi>0$ small \textcolor{red}{(depending on $\vep$ and the points)} so that
%\begin{equation}%\label{eq:bddoscillation}
%\xi  + \osc(r,\xi) \, \max \{n_i : 1\le i \le k\}) <\vep.
%\end{equation}
%\vspace{.5cm}

Now we shall use the
 gluing orbit property  %specification property
for $f$ with the proximity $\xi$. More precisely, if
 $N(\xi)$ is given by the  gluing orbit property  %specification property
 for $f$ then there exists $x\in M$ that shadows
 the pieces of the orbits of the points $x_i$ during $n_i + 1$ iterates with a time lag of
  at most  $N(\xi)$ iterates. %exactly
More precisely, there are $\tilde{p}_i \le N(\xi)$, $1\le i \le k$, and $x\in M$ so that
$
d (f^j(x),f^j(x_1) ) \leq \xi
$
for every $0\leq j \leq n_1 +1$ and
$
d (f^{j+n_1+\tilde{p}_1+\dots +n_{i-1}+\tilde{p}_{i-1} +(i-1)}(x) \,,\, f^j(x_i) )
        \leq \xi
$
for every $2\leq i\leq k$ and $0\leq j\leq n_i +1$.
Choose $T(\vep):=T(\xi)= ( N(\xi)+2 )\, \sup \rho$. Observe that $T(\xi)$ depends
only on $\xi$ (hence only depending on $\vep$) and the upper bound for the roof function. Set $s=s_1$.

Before giving the full details of the proof let us make some comments to illustrate the difficulties involved.
The proof of the theorem consists of proving that the trajectory of the point $(x,s)$ under
the action of the suspension semiflow follows closely the pieces of orbit of the prescribed points $(x_i,s_i)$
with a control on the time in between. At each moment of the shadowing process one needs a control on
the lap number involving either the point $x$ or the points $x_i$. Since the lap number corresponding to $x$ and the one
for some $x_i$ may differ by one, there are at most $18^{k-1}$ cases to consider.
We will explicit the key estimates in the case where $k=2$, which encloses all the difficulties of the general case and
where the notation is greatly simplified. The general case involves a completely analogous but much more technical computation using the ideas from the case $k=2$. For that purpose, in the remaining we will prove the following:

\vspace{.15cm}
\noindent {\bf Claim:}
$
d_\rho( X_{t}  ({x}, s) , X_t(x_1,s_1) ) <\vep
$
for every  $t \in [0, t_1]$ and there exists $0\le p_1 \le T(\vep)$ so that
$
d_\rho( X_{t+t_1+p_1}  ({x}, s_1) , X_t(x_2,s_2) ) <\vep
$
for every  $t \in [0, t_2]$.
\vspace{.15cm}

%Consider the points $(x, s), (x, s_1)\in M_\rho$.
Since $s=s_1$, for every $t\in [0,t_1]$ one can write
$$
X_{t}  (x, s_1)= \Big( f^j(x), s_1+t - \sum_{j=0}^{j-1} \rho(f^j(x)) \Big)
$$
and
$$
X_{t}  (x_1, s_1)= \Big( f^{j_1}(x_1), s_1+t - \sum_{j=0}^{j_1-1} \rho(f^j(x_1)) \Big),
$$
where $j=j(x,s_1,t)\in \mathbb N_0$ and $j_1=j_1(x_1,s_1,t) \in\mathbb N_0$ are uniquely determined by
\begin{equation}\label{equa5}
\sum_{i=0}^{j-1} \rho(f^i(x)) \le s_1+t < \sum_{i=0}^{j} \rho(f^i(x))
    \quad\text{and}\quad
    \sum_{i=0}^{j_1-1} \rho(f^i(x_1)) \le s_1+t < \sum_{i=0}^{j_1} \rho(f^i(x_1)).
\end{equation}
By the choice of $\xi$ it follows that $C_\xi \ll \inf_{x\in M} \rho(x)$ and so
$|j(x,s_1,t) - j_1(x,s_1,t)| \le 1$ for every $t \in [0, t_1]$.
Fix $t\in [0,t_1]$. We can estimate the $d_{\rho}$-distance according to the following three prototypical cases:
\begin{itemize}
 \item[(i)] if $j=j(x,s_1,t) =  j_1(x,s_1,t)$ then estimating the distance from above by the natural
 	horizontal and vertical segments (see Figure~\ref{fig:1} below)  it follows that
\begin{align}
    d_{\rho} & ( X_{t}  (x, s_1) , X_t(x_1,s_1) ) \nonumber \\
         & = d_{1}\Big( \Big( f^j(x), \frac{s_1+t - \sum_{i=0}^{j-1} \rho(f^i(x))}{\rho(f^{j}(x))} \Big) , \Big( f^{j_1}(x_1),
         		\frac{s_1+t - \sum_{i=0}^{j_1-1} \rho(f^i(x_1))}{\rho(f^{j_1}(x_1))} \Big) \Big) \nonumber  \\
	& \le d_{1}\Big( \big( f^j(x), \frac{s_1+t - \sum_{i=0}^{j_{1}-1} \rho(f^i(x))}{\rho(f^{j}(x))} \big) , \big( f^{j}(x_1),
		\frac{s_1+t - \sum_{i=0}^{j_1-1} \rho(f^i(x))}{\rho(f^{j_1}(x))} \big) \Big) \tag{$A_1$} \\
	& +      d_{1}\Big( \big( f^{j}(x_1), \frac{s_1+t - \sum_{i=0}^{j_1-1} \rho(f^i(x))}{\rho(f^{j_1}(x))} \big) , \big( f^{j_1}(x_1),
		\frac{s_1+t - \sum_{i=0}^{j_1-1} \rho(f^i(x_1))}{\rho(f^{j_1}(x_1))} \big) \Big) \tag{$A_2$}
\end{align}
and consequently,
\begin{align*}
    d_{\rho}( X_{t}  (x, s_1) , X_t(x_1,s_1) )
	& \le      (1 - \frac{s_1+t - \sum_{i=0}^{j-1} \rho(f^i(x))}{\rho(f^{j}(x))}) \; d(f^{j}(x) , f^{j}(x_{1})) \\
	& + \frac{s_1+t - \sum_{i=0}^{j-1} \rho(f^i(x))}{\rho(f^{j}(x))}  \; d(f^{j+1}(x) , f^{j+1}(x_{1})) \\
	& +     \Big| \frac{s_1+t - \sum_{i=0}^{j_1-1} \rho(f^i(x))}{\rho(f^{j_1}(x))} - \frac{s_1+t - \sum_{i=0}^{j_1-1} \rho(f^i(x_1))}{\rho(f^{j_1}(x_1))} \Big|.
\end{align*}
Since points in the same dynamical ball for $f$ remain up to distance $\xi$ along the prescribed piece of orbit, the sum of the
first two terms in the right hand side above are smaller than $\xi$. We shall bound differently
the third summand in the right hand side above, which we will denote by $(*)$.
 %according to whether $\rho$ is uniformly
%continuous.
%or $\log \rho$ is H\"older continuous.
By the choice of $\xi$ and uniform continuity of $\rho$, by triangular inequality,
\begin{align*}
 (*)  & \leq
     \Big| \frac{s_1+t - \sum_{i=0}^{j_1-1} \rho(f^i(x)) - s_1-t + \sum_{i=0}^{j_1-1} \rho(f^i(x_1))}{\rho(f^{j_1}(x))} \Big| \\
    & + \frac{\big(s_1+t - \sum_{i=0}^{j_1-1} \rho(f^i(x_1)) \big)}{\rho(f^{j_1}(x)) \cdot
    	\rho(f^{j_1}(x_{1}))}|\rho(f^{j_1}(x)) - \rho(f^{j_1}(x_{1}))|
	% \leq \frac{K_{\xi}}{\inf \rho} + \frac{\vep \inf \rho}{3\inf \rho}
    \leq \frac{C_{\xi}}{\inf \rho} + \frac{C_{\xi}}{\inf \rho}.
\end{align*}
%\textcolor{magenta}{In second case, $\log \rho$ is H\"older and so $e^{-\xi} \leq \frac{\rho(f^i(x_1))}{\rho(f^i(x))} \leq e^{\xi}$
%whenever $d_0(z,w) < \xi$}. If this is the case, one can estimate
%\begin{align*}
% A &
% 	\leq \Big|\frac{s_1+t - \sum_{i=0}^{j_1-1} \rho(f^i(x))}{e^{\xi}\rho(f^{j_1}(x_{1}))} - \frac{s_1+t - \sum_{i=0}^{j_1-1} \rho(f^i(x_1))}{\rho(f^{j_1}(x_{1}))} \Big| \\
%	& \le  \Big|\frac{s_1+t - \sum_{i=0}^{j_1-1} \rho(f^i(x)) - s_1-t + \sum_{i=0}^{j_1-1} \rho(f^i(x_1))}{e^{\xi}\rho(f^{j_1}(x_{1}))} 	
%	\Big| \\
%	& +  \frac{(1 - e^{\xi})}{e^{\xi}\rho(f^{j_1}(x_{1}))}\Big| (s_1+t - \sum_{i=0}^{j_1-1} \rho(f^i(x_1)) \Big| \\
%	& \le  \frac{K_{\xi}}{e^{\xi}\inf \rho} + \frac{1- e^{\xi}}{e^{\xi}}.
%\end{align*}
By choice of $\xi>0$ we get $d_{\rho}( X_{t}  (x, s_1) , X_t(x_1,s_1) ) < \vep$ as required.
\begin{figure}[htbp]
 \begin{center}
  \includegraphics[width=16cm]{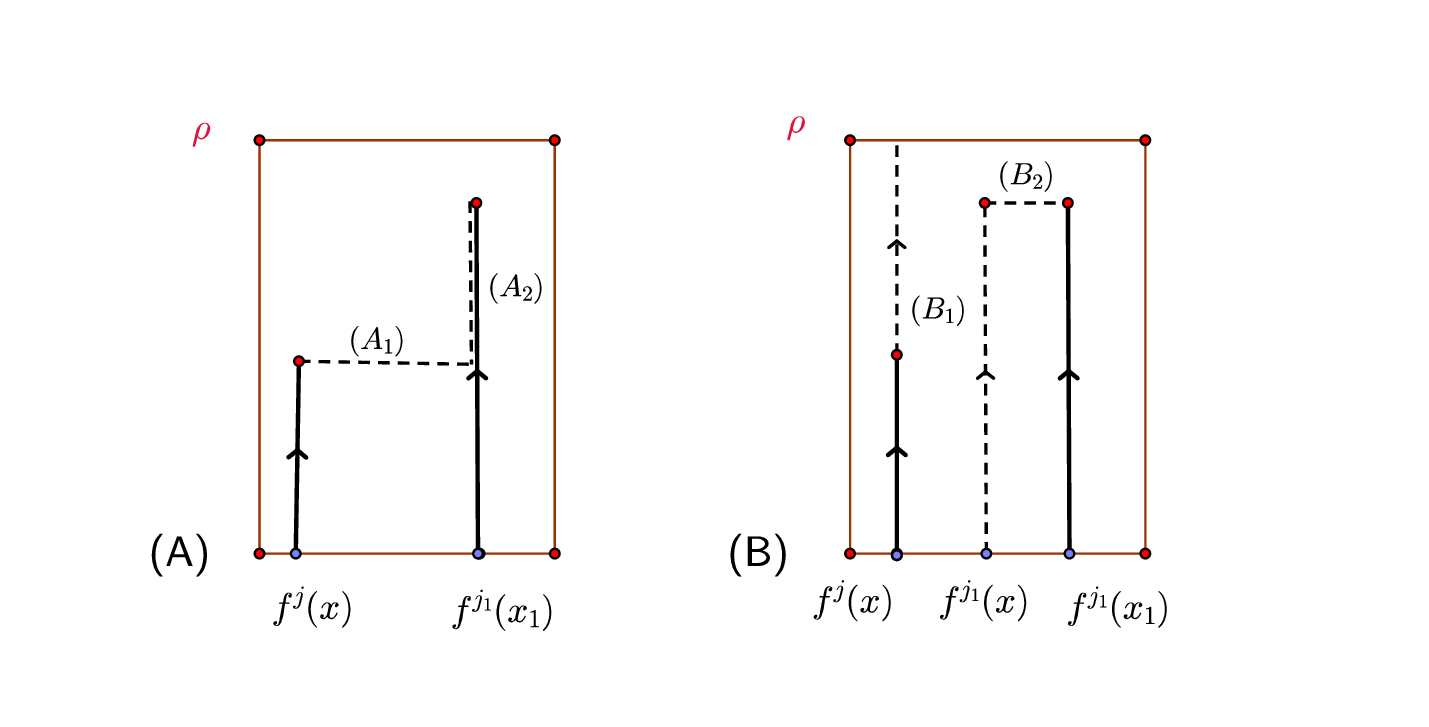}
 \end{center}
 \caption{Schematic description of the (dotted) vertical and horizontal segments used to estimate the Bowen-Walters distance:
 (A) corresponds to case (i) above; (B) corresponds to cases (ii) and (iii) below.}
 \label{fig:1}
\end{figure}

\item[(ii)] The second case to consider is the case $j=j(x,s_1,t) =  j_1(x,s_1,t) + 1$.  Noticing that $f^{j}(x)$ and $f^{j_1}(x)$
	are consecutive elements of the same orbit, we get
\begin{align}
    d_{\rho} & ( X_{t}  (x, s_1) , X_t(x_1,s_1) ) \nonumber \\
         & = d_{1}\Big( \big( f^j(x), \frac{s_1+t - \sum_{i=0}^{j-1} \rho(f^i(x))}{\rho(f^{j}(x))} \big) \,,\, \big( f^{j_1}(x_1), \frac{s_1+t -
         	\sum_{i=0}^{j_1-1} \rho(f^i(x_1))}{\rho(f^{j_1}(x_1))} \big) \Big) \nonumber \\
	& \le   d_{1}\Big( \big( f^{j_{1}}(x_{1}), \frac{s_1+t - \sum_{i=0}^{j_{1}-1} \rho(f^i(x_{1}))}{\rho(f^{j_{1}}(x_{1}))} \big) \,,\, \big( f^{j_{1}}(x), \frac{s_1+t - \sum_{i=0}^{j_{1}-1} \rho(f^j(x_{1}))}{\rho(f^{j_{1}}(x_{1}))} \big) \Big) \tag{$B_2$} \\
	& +     d_{1}\Big( \big( f^{j_{1}}(x), \frac{s_1+t - \sum_{i=0}^{j_{1}-1} \rho(f^j(x_{1}))}{\rho(f^{j_{1}}(x_{1}))} \big) \,,\, \big( f^j(x), \frac{s_1+t - \sum_{i=0}^{j-1} \rho(f^i(x))}{\rho(f^{j}(x))} \big) \Big) \tag{$B_1$} \\
	& \leq     \Big(1 - \frac{s_1+t - \sum_{i=0}^{j_{1}-1} \rho(f^i(x_{1}))}{\rho(f^{j_{1}}(x_{1}))} \Big) \; d(f^{j_{1}}(x_{1}) , f^{j_{1}}(x)) 	\nonumber \\
	& +\; \frac{s_1+t - \sum_{i=0}^{j_{1}-1} \rho(f^i(x_{1}))}{\rho(f^{j_{1}}(x_{1}))} \; d(f^{j_{1}+1}(x_{1}) , f^{j_{1}+1}(x)) + (**)
	\nonumber
\end{align}
where  $ (**) := \Big| \big(1 - \frac{s_1+t - \sum_{i=0}^{j_1-1} \rho(f^i(x_{1}))}{\rho(f^{j_1}(x_{1}))} \big)
+ \frac{s_1+t - \sum_{i=0}^{j-1} \rho(f^i(x))}{\rho(f^{j}(x))} \Big|$ (see Figure~\ref{fig:1} above).
Since $x$ was chosen so that its orbit to approximates the orbit of $x_1$ during the fist $n_1+1$ iterates
then the sum of the first two terms is smaller than $\xi$.
Since  the two terms involved in the absolute value are positive and $j=j(x,s_1,t) =  j_1(x,s_1,t) + 1$ it follows from relations (\ref{equa5}) that
    $$
     (**) = \frac{-s_{1} - t + \sum_{i=0}^{j_1} \rho(f^i(x_{1}))}{\rho(f^{j_1}(x_{1}))}
     	+ \frac{s_1+t - \sum_{i=0}^{j-1} \rho(f^i(x))}{\rho(f^{j}(x))}
	\leq \frac{C_{\xi}}{\inf \rho}
	+ \frac{C_{\xi}}{\inf \rho}.
    $$
     Hence, we obtain that $d_{\rho}( X_{t}  (x, s_1) , X_t(x_1,s_1) ) < \vep$.
\item[(iii)] If $j=j(x,s_1,t) =  j_1(x,s_1,t) - 1$ the computations are completely analogous to (ii) interchanging the roles of $x_1$ and $x$.
\end{itemize}

After the choice of the point $(x,s)$, partially determined by the gluing orbit property
for $f$ and also by taking $s=s_1$, we claim that one can prove that the second assertion in the Claim
is also satisfied. For each of the previous situations (i)-(iii) above (at time $t_1$)
we will subdivide the proof in three additional cases, corresponding to the relative position of the lap number of $x$
and $x_2$. This will be made precise in the remaining of this section.

\medskip
First assume \textit{case (i)} above holds at time $t_1$, that is $j_1:=j(x,s_1,t_1) =  j_1(x_1,s_1,t_1)$ (c.f. Figure~\ref{fig:2}
below). In other words
$$
X_{t_1}  ({x},  s)= \Big( f^{j_1}(x),  s_1+t_1 - \sum_{i=0}^{j_1-1} \rho(f^i(x)) \Big)
	\; \text{and}\;
	X_{t_1}  (x_1,  s_1)= \Big( f^{j_1}(x),  s_1+t_1 - \sum_{i=0}^{j_1-1} \rho(f^i(x)) \Big)
$$
where
\begin{equation*}
\sum_{i=0}^{j_1-1} \rho(f^i(x)) \le s_1+t_1 < \sum_{i=0}^{j_1} \rho(f^i(x))
    \quad\text{and}\quad
    \sum_{i=0}^{j_1-1} \rho(f^i(x_1)) \le s_1+t_1 < \sum_{i=0}^{j_1} \rho(f^i(x_1)).
\end{equation*}

\begin{figure}[htbp]
 \begin{center}
  \includegraphics[width=8cm, height=7cm]{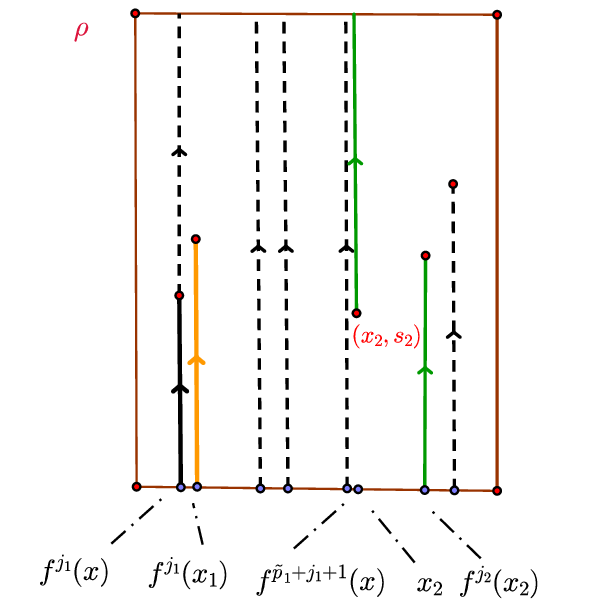}
 \end{center}
 \caption{The dotted line represents the piece of the trajectory of $x$ shadowing the piece of trajectory
 $X_t(f^{j_1}(x_1),0)$ for $t\in [0, s_1 + t_1 - \sum_{i=0}^{j_1-1} \rho(f^i(x))]$ and, after some
 time $p_1$, shadows the piece of trajectory
 $X_t(x_2,s_2)$ for a time $t\in [0, t_2]$. }
 \label{fig:2}
\end{figure}

In this case take
    $$
    p_1
      =
    \begin{cases}
    \begin{array}{ll}
    s_2
    + \sum_{i=0}^{\tilde{p}_{1}-1} \rho(f^{j_1+i}(x))
    - [s_1 + t_1 - \sum_{i=0}^{j_1-1} \rho(f^i(x))],
    & \text{ if } s_{2} \leq \rho(f^{\tilde{p}_{1}+ j_{1}}(x))
    \\
    (s_2 - C_{\xi})
    +\sum_{i=0}^{\tilde{p}_{1}-1} \rho(f^{j_1+i}(x))
    - [s_1 + t_1 - \sum_{i=0}^{j_1-1} \rho(f^i(x))],
    & \text{ otherwise. }
    \end{array}
   \end{cases}
    $$
In both cases above it is clear that $|p_1| \le (\tilde p_1 + 2) \sup \rho \le (N(\xi) + 2) \sup \rho = T(\vep)$.
Now one can estimate $d_{\rho}( X_{t+p_1+t_1}(x,s_1), X_t(x_2,s_2))$ according to the relative position
of lap numbers.

If $s_{2} \leq \rho(f^{\tilde{p}_{1}+ j_{1}}(x))$ then $X_{p_1+t_1} (x,s)=(f^{j_1+\tilde p_1}(x), s_2)$.
For any $t\in [0,t_2]$ set by, some abuse of notation,
%{\color{red}
%one can write
%$$
%X_{t}  (x, s_1)= \Big( f^j(x), s_1+t - \sum_{j=0}^{j-1} \rho(f^j(x)) \Big)
%$$
%and
%$$
%X_{t}  (x_1, s_1)= \Big( f^{j_1}(x_1), s_1+t - \sum_{j=0}^{j_1-1} \rho(f^j(x_1)) \Big),
%$$
%where }
$j=j(f^{j_1+\tilde p_1}(x),s_2,t) \in \mathbb N_0$ and $j_2=j_2(x_2,s_2,t) \in\mathbb N_0$ which are uniquely determined by
\begin{equation}\label{equa6}
\sum_{i=0}^{j-1} \rho(f^{i+j_1+\tilde p_1}(x)) \le s_2+t < \sum_{i=0}^{j} \rho(f^{i+j_1+\tilde p_1}(x))
    \quad\text{and}\quad
    \sum_{i=0}^{j_2-1} \rho(f^i(x_2)) \le s_2+t < \sum_{i=0}^{j_2} \rho(f^i(x_2)).
\end{equation}
These lap numbers satisfy $|j(f^{j_1+\tilde p_1}(x),s_2,t) - j_2(x_2,s_2,t)| \le 1$ for every $t \in [0, t_2]$.
Subdividing the later in three cases, when $j=j_2$, $j=j_2+1$ and $j=j_2-1$, we can deduce similarly as before that
$d_{\rho}\Big( X_{t+p_1+t_1}(x,s_1), X_t(x_2,s_2) \Big) < \vep$ for every $t\in [0,t_2]$.

If $s_{2} > \rho(f^{\tilde{p}_{1}+ j_{1}}(x))$ then $X_{p_1+t_1} (x,s)=(f^{j_1+\tilde p_1}(x), s_2-C_\xi)$.
For any $t\in [0,t_2]$ set
$j=j(f^{j_1+\tilde p_1}(x),s_2-C_\xi,t) \in \mathbb N_0$ uniquely determined by
\begin{equation}\label{equa7}
\sum_{i=0}^{j-1} \rho(f^{i+j_1+\tilde p_1}(x)) \le (s_2 - C_\xi) +t < \sum_{i=0}^{j} \rho(f^{i+j_1+\tilde p_1}(x))
\end{equation}
and $j_2=j_2(x_2,s_2,t) \in\mathbb N_0$ determined by ~\eqref{equa6}. In the case that $j=j_2$,
%
%PQ OS LAP NUMBERS NAO DIFEREM POR K. C_XI?
%
\begin{align*}
    d_{\rho} & ( X_{t+p_1+t_1}(x,s), X_t(x_2,s_2)) \nonumber \\
         & = d_{1}\big( \big( f^{j+j_{1}+\tilde{p}_{1}}(x), \frac{s_2 - C_{\xi} + t - \sum_{i=0}^{j-1} \rho(f^{i+j_{1}+\tilde{p}_{1}}(x))}{\rho(f^{j+j_{1}+\tilde{p}_{1}}(x))} \big) , \big( f^{j_2}(x_2),
         		\frac{s_2+t - \sum_{i=0}^{j_2-1} \rho(f^i(x_2))}{\rho(f^{j_2}(x_2))} \big) \big)  \\
	& \le d_{1}\big( \big( f^{j+j_{1}+\tilde{p}_{1}}(x), \frac{s_2 - C_{\xi} +t - \sum_{i=0}^{j-1} \rho(f^{i+j_{1}+\tilde{p}_{1}}(x))}	
		{\rho(f^{j + j_{1}+\tilde{p}_{1}}(x))} \big) ,  \big( f^{j_{2}}(x_2),
		\frac{s_2 - C_{\xi} +t - \sum_{i=0}^{j-1} \rho(f^{i+j_{1}+\tilde{p}_{1}}(x))}{\rho(f^{j+ j_{1}+\tilde{p}_{1}}(x))} \big) \big)  \\
	& +      d_{1}\Big( \big( f^{j_{2}}(x_2),
		\frac{s_2 - C_{\xi} +t - \sum_{i=0}^{j-1} \rho(f^{i+j_{1}+\tilde{p}_{1}}(x))}{\rho(f^{j+ j_{1}+\tilde{p}_{1}}(x))} \big) , \big( f^{j_2}(x_2),
		\frac{s_1+t - \sum_{i=0}^{j_2-1} \rho(f^i(x_2))}{\rho(f^{j_2}(x_2))} \big) \Big)
\end{align*}
and consequently,
\begin{align*}
    d_{\rho}( X_{t+p_1+t_1}(x,s), X_t(x_2,s_2))
	& \le      (1 - \frac{s_2 - C_{\xi} +t - \sum_{i=0}^{j-1} \rho(f^{i+j_{1}+\tilde{p}_{1}}(x))}{\rho(f^{j+ j_{1}+\tilde{p}_{1}}(x))} )
		\; d(f^{j + j_{1}+\tilde{p}_{1}}(x) , f^{j_{2}}(x_{2})) \\
	& + \frac{s_2 - C_{\xi} +t - \sum_{i=0}^{j-1} \rho(f^{i+j_{1}+\tilde{p}_{1}}(x))}{\rho(f^{j+ j_{1}+\tilde{p}_{1}}(x))}  \;
		d(f^{j +1+ j_{1}+\tilde{p}_{1}}(x) , f^{j_{2}+1}(x_{2})) \\
	& +     \Big| \frac{s_2 - C_{\xi} +t - \sum_{i=0}^{j-1} \rho(f^{i+j_{1}+\tilde{p}_{1}}(x))}{\rho(f^{j + j_{1}+\tilde{p}_{1}}(x))} - \frac{s_2+t - \sum_{i=0}^{j_2-1} \rho(f^i(x_2))}{\rho(f^{j_2}(x_2))} \Big|.
\end{align*}
Since $j=j_2$ and points in the same dynamical ball for $f$ remain up to distance $\xi$ along the prescribed piece of orbit,
the sum of the
first two terms in the right hand side above are smaller than $\xi$. We shall bound differently
the third summand in the right hand side above, which we will denote by $(***)$.
By triangular inequality,
\begin{align*}
 (***)  & \leq
     \Big| \frac{s_2 - C_{\xi} +t - \sum_{i=0}^{j-1} \rho(f^{i+j_{1}+\tilde{p}_{1}}(x)) - s_2-t + \sum_{i=0}^{j_2-1} \rho(f^i(x_2))}{\rho(f^{j+j_{1}+\tilde{p}_{1}}(x))} \Big| \\
    & + \frac{\big(s_2+t - \sum_{i=0}^{j_2-1} \rho(f^i(x_2)) \big)}{\rho(f^{j+j_{1}+\tilde{p}_{1}}(x)) \cdot
    	\rho(f^{j_2}(x_{2}))}|\rho(f^{j+j_{1}+\tilde{p}_{1}}(x)) - \rho(f^{j_2}(x_{2}))|
    \leq \frac{2C_{\xi}}{\inf\rho} + \frac{C_{\xi}}{\inf \rho} .
\end{align*}
%\footnote{\textcolor{red}{continuo preocupado com este $2$ que aparece acima a vermelho, porque com $k$ pontos poderia %transformar-se em $k C_\xi$... , concorda?}}
%
%%%%%%%%%%%%%%%%%%%%%%%%%%%%%%%%%%%%
%uma das estimativas antigas
%\begin{align*}
%d_{\rho}( X_{t+p_1+t_1}(x,s), X_t(x_2,s_2))
%	& = d_{\rho}( X_t (f^{j_1+\tilde p_1}(x), s_2-C_\xi) , X_t(x_2,s_2) ) \\
%	& \le d_{\rho}( X_t (f^{j_1+\tilde p_1}(x), s_2-C_\xi) , X_t(x_2,s_2-C_\xi) ) \\
%	& + d_{\rho}( X_t(x_2,s_2-C_\xi), X_t(x_2,s_2) ) \\
%	& \le \Big(1 - \frac{s_{2}-C_\xi}{\rho(f^{j+\tilde{p}_{1}+1}(x))} \Big)d\big(f^{j+\tilde{p}_{1}+j_1}(x) , f^j(x_{2}) \big) \\
	%
%	& +  \frac{s_{2}-C_\xi}{\rho(f^{j+\tilde{p}_{1}+j_1}(x))} d\big(f^{j+1+\tilde{p}_{1}+j_1}(x) , f^{j+1}(x_{2}) \big) \\
	%
%	& + \Big|\frac{s_{2}}{\rho(f^j(x_{2}))} - \frac{s_{2}-C_\xi}{\rho(f^{j+\tilde{p}_{1}+j_1}(x))} \Big| \\
%	& +
%\end{align*}
%which is bounded by $\xi + C$, where $C := \Big|\frac{s_{2}}{\rho(x_{2})} - \frac{s_{2}-C_\xi}{\rho(f^{n_{1}+\tilde{p}_{1}+1}(x))} \Big|$.
%By choice of $\xi$ and uniform continuity of $\rho$ then $C \leq \frac{\vep}3$.}
%$$
%C \leq
%	\frac{\sup \rho}{(\inf \rho)^2}  |\rho(x_{2}) - \rho(f^{n_{1}+\tilde{p}_{1}+1}(x))|
%	\leq \frac{\vep}3.
%$$	
The estimates in the case $j=j_2-1$ and $j=j_2+1$ are obtained similarly.
In consequence $d_{\rho}\big( X_{t+p_1+t_1}(x,s), X_t(x_2,s_2) \big) < \vep$ for every $t\in [0,t_2]$.

\medskip
Now assume  \textit{case (ii)} above at time $t_1$, that is, $j(x,s,t_1) =  j_1 + 1$ with
$j_1=j_1(x_1,s_1,t)$ (see Figure~\ref{fig:3}
below).
\begin{figure}[htbp]
 \begin{center}
  \includegraphics[width=8cm, height=8cm]{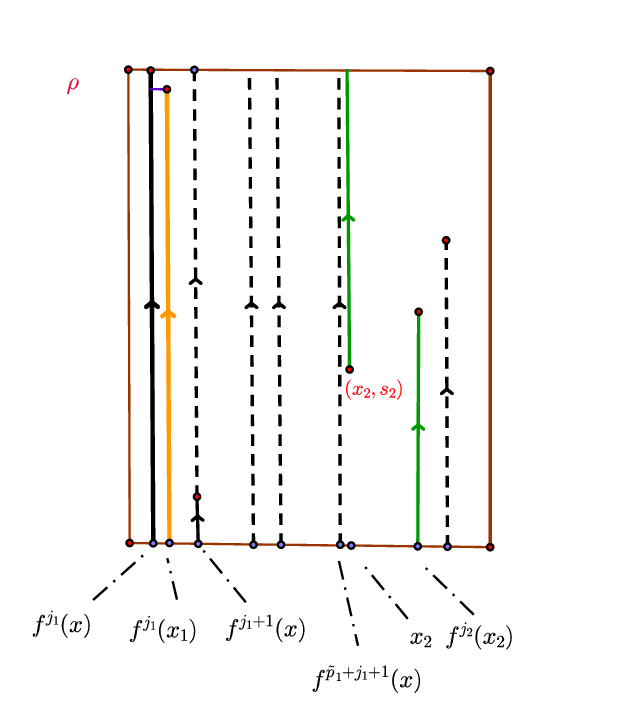}
 \end{center}
 \caption{
Schematic picture where the dotted line represents the piece of the trajectory of $x$ shadowing the piece of trajectory
 $X_t(f^{j_1+1}(x_1),0)$ for $t\in [0, s_1 + t_1 - \sum_{i=0}^{j_1} \rho(f^i(x))]$ and, after some
 time $p_1$, follows the piece of the trajectory $X_t(x_2,s_2)$ for $t\in [0, t_2]$.}
% (case $j(x,s,t_1) =  j_1(x_1,s_1,t) + 1$).
 \label{fig:3}
\end{figure}
If this is the case take
    $$
    p_1
      =
    \begin{cases}
    \begin{array}{ll}
    s_2
    + \sum_{i=0}^{\tilde{p}_{1}-1} \rho(f^{j_1+1+i}(x))
    - [s_1 + t_1 - \sum_{i=0}^{j_1} \rho(f^i(x))],
    & \text{ if } s_{2} \leq \rho(f^{\tilde{p}_{1}+ j_{1}+1}(x))
    \\
    (s_2 - C_{\xi})
    + \sum_{i=0}^{\tilde{p}_{1}-1} \rho(f^{j_1+1+i}(x))
    - [s_1 + t_1 - \sum_{i=0}^{j_1} \rho(f^i(x))],
    & \text{ otherwise. }
    \end{array}
   \end{cases}
    $$
As above, $|p_1| \le T(\vep)$ in both cases. If $s_{2} \leq \rho(f^{\tilde{p}_{1}+ j_{1}+1}(x))$ then computations completely identical to case (ii) proving
$d_{\rho}\big( X_{t+p_1+t_1}(x,s), X_t(x_2,s_2) \big) < \vep$ for every $t\in [0,t_2]$. In the case that
$s_{2} > \rho(f^{\tilde{p}_{1}+ j_{1}+1}(x))$ it follows that
\begin{align}
    d_{\rho} & \big( X_{t+p_1+t_1}(x,s), X_t(x_2,s_2) \big) \nonumber \\
         & = d_{1}\big( \big( f^{j+j_{1}+\tilde{p}_{1}}(x), \frac{s_2 - C_{\xi} + t - \sum_{i=0}^{j-1} \rho(f^{i+j_{1}+\tilde{p}_{1}}(x))}{\rho(f^{j+j_{1}+\tilde{p}_{1}}(x))} \big) , \big( f^{j_2}(x_2),
         		\frac{s_2+t - \sum_{i=0}^{j_2-1} \rho(f^i(x_2))}{\rho(f^{j_2}(x_2))} \big) \big) \nonumber \\
	& \le   d_{1}\Big( \big( f^{j_{2}}(x_{2}), \frac{s_2+t - \sum_{i=0}^{j_{2}-1} \rho(f^i(x_{2}))}{\rho(f^{j_{2}}(x_{2}))} \big) \,,\, \big( f^{j_{2}+j_{1}+\tilde{p}_{1}}(x), \frac{s_2+t - \sum_{i=0}^{j_{2}-1} \rho(f^j(x_{2}))}{\rho(f^{j_{2}}(x_{2}))} \big) \Big)  \nonumber\\
	& +     d_{1}\Big( \big( f^{j_{2}+j_{1}+\tilde{p}_{1}}(x), \frac{s_2+t - \sum_{i=0}^{j_{2}-1} \rho(f^j(x_{2}))}{\rho(f^{j_{2}}(x_{2}))} \big) \,,\, \big( f^{j+j_{1}+\tilde{p}_{1}}(x), \frac{s_2 - C_{\xi} + t - \sum_{i=0}^{j-1} \rho(f^{i+j_{1}+\tilde{p}_{1}}(x))}{\rho(f^{j+j_{1}+\tilde{p}_{1}}(x))} \big) \Big)  \nonumber\\
	& \leq     \Big(1 - \frac{s_2+t - \sum_{i=0}^{j_{2}-1} \rho(f^j(x_{2}))}{\rho(f^{j_{2}}(x_{2}))} \Big) \; d(f^{j_{2}}(x_{2}) , f^{j_{2}+j_{1}+\tilde{p}_{1}}(x)) 	\nonumber \\
	& +\; \frac{s_2+t - \sum_{i=0}^{j_{2}-1} \rho(f^j(x_{2}))}{\rho(f^{j_{2}}(x_{2}))} \; d(f^{j_{2}+1}(x_{2}) , f^{j_{2}+j_{1}+\tilde{p}_{1}+1}(x)) + (\star)
	\nonumber
\end{align}
where  $ (\star) := \Big| \big(1 - \frac{s_2+t - \sum_{i=0}^{j_{2}-1} \rho(f^j(x_{2}))}{\rho(f^{j_{2}}(x_{2}))} \big)
+ \frac{s_2 - C_{\xi} + t - \sum_{i=0}^{j-1} \rho(f^{i+j_{1}+\tilde{p}_{1}}(x))}{\rho(f^{j+j_{1}+\tilde{p}_{1}}(x))} \Big|$.
By choose of $x$ follows that the sum of the first two terms is smaller than $\xi$.
Since  the two terms involved in the absolute value are positive and $j=j(f^{j_1+\tilde p_1}(x),s_2-C_\xi,t) =  j_2(x_{2},s_2,t) + 1$ we have
    $$
     (\star) =  \Big| - \frac{s_2+t - \sum_{i=0}^{j_{2}} \rho(f^j(x_{2}))}{\rho(f^{j_{2}}(x_{2}))}
+ \frac{s_2 - C_{\xi} + t - \sum_{i=0}^{j-1} \rho(f^{i+j_{1}+\tilde{p}_{1}}(x))}{\rho(f^{j+j_{1}+\tilde{p}_{1}}(x))} \Big|
	\leq \frac{2C_{\xi}}{\inf \rho} + \frac{C_{\xi}}{\inf \rho}.
    $$
     Hence, we obtain that $d_{\rho}( X_{t}  (x, s_1) , X_t(x_1,s_1) ) < \vep$.

\medskip
If \textit{Case (iii)} holds, for which $j(x,s,t_1) =  j_1(x_1,s_1,t) - 1$, again completely analogous to Case (ii)
with a modification on the definition of $p_1$ which must be replaced by
    $$
    p_1
      =
    \begin{cases}
    \begin{array}{ll}
    s_2
    + \sum_{i=0}^{\tilde{p}_{1}-1} \rho(f^{j_1-1+i}(x))
    - [s_1 + t_1 - \sum_{i=0}^{j_1-2} \rho(f^i(x))],
    & \text{ if } s_{2} \leq \rho(f^{\tilde{p}_{1}+ j_{1}-1}(x))
    \\
    (s_2 - C_{\xi})
    + \sum_{i=0}^{\tilde{p}_{1}-1} \rho(f^{j_1-1+i}(x))
    - [s_1 + t_1 - \sum_{i=0}^{j_1-2} \rho(f^i(x))],
    & \text{ otherwise. }
    \end{array}
   \end{cases}
    $$
The remaining estimates for the finite time shadowing necessary for proving the gluing orbit property are identical to the ones
we have obtained above and for that reason we shall omit the details.
This completes the proof of the theorem.

%\begin{remark}
%Since the requirement of the theorem on the base dynamics to satisfy a gluing orbit property then the later result
%applies for suspension flows of transitive but non topologically mixing subshifts of finite type.
%\end{remark}

%%%%%%%%%%%%%%%%%%%%%%%%%%%%
\subsection{Proof of Theorem~\ref{thm:nugluing1}}

Assume that $f: M \to M$ satisfies the gluing orbit property. Let $\mu$ be an $f$-invariant ergodic
probability measure  and that the roof function $\rho: M \to \mathbb R_0^+$ is integrable.
Fix $\vep>0$. Consider arbitrary points $(x_1,s_1), (x_2,s_2), \dots, (x_k,s_k) \in M_r$ in a $\bar\mu$-full measure set
in such a way that $\lim_{\xi\to 0} C_\xi(x_i)=0$ for every $1\le i \le k$ and consider arbitrary times
$t_1, \dots, t_k \ge 0$. Associated to $(x_i, s_i)$ and $t_i$ consider the dynamical balls
$B(x_i, n_i,\vep) \subset M$, where $n_i=n_i(x_i, s_i,t_i)\ge 1$ is determined by equation \eqref{eq:turns}.
Let $\xi>0$ be such that $\xi+C_\xi(x_i)<\vep$ for every $1\le i \le k$ and let $N(\xi)\ge 1$ be given by the
gluing orbit property for $f$.
%
%OLD VERSION
%For any fixed $\vep>0$, let $\xi>0$ be such that $\xi+C_\xi<\vep$ and let $N(\xi)\ge 1$ be given by the
%specification property for $f$.
%
%Consider arbitrary points $(x_1,s_1), (x_2,s_2), \dots, (x_k,s_k) \in M_r$ and times
%$t_1, \dots, t_k \ge 0$. Associated to $(x_i, s_i)$ and $t_i$ consider the dynamical balls
%$B(x_i, n_i,\vep) \subset M$, where $n_i=n_i(x_i, s_i,t_i)\ge 1$ is determined by equation \eqref{eq:turns}.
Thus, there exists $x\in M$ and $\tilde{p}_i \le N(\xi)$, $1\le i \le k$, so that
$
d (f^j(x),f^j(x_1) ) \leq \xi
$
for every $0\leq j \leq n_1 +1$ and
$
d (f^{j+n_1+\tilde{p}_1+\dots +n_{i-1}+\tilde{p}_{i-1}}(x) \,,\, f^j(x_i) )
        \leq \xi
$
for every $2\leq i\leq k$ and $0\leq j\leq n_i +1$.
 %Taking $p_i=N(\xi)$ for every $i$, by the specification property for $f$
%there exists $x\in M$ such that
%$
%d ( f^j(x),f^j(x_1) ) \leq \xi
 %   \quad \text{and} \quad
  %  d ( f^{j+n_1+p_1+\dots +n_{i-1}+p_{i-1}}(x) \;,\; f^j(x_i) ) \leq \xi
%$
%for every $2\leq i\leq k$ and $0\leq j\leq n_i$.
%
The proof follows the same strategy as in Theorem~\ref{thm:gluingAA} with due care by the fact $\rho$ is not
necessarily bounded but $\rho\in L^1(\mu)$. Take
$$
T((x_i,s_i),t_i,\xi)
    := \sum_{j=0}^{\tilde{p}_{i}+2} \rho( f^{j+n_i-1}(x_i)) + C_\xi(x_i)
$$
(where $\xi>0$ depends on $\vep$)
and decompose the terms as the pieces of the orbits of $x_i$ and the terms corresponding
to the specified time lags $\tilde{p}_{i}$ as follows:
$$
\frac{1}{n_i + \tilde{p}_{i}}
     \sum_{j=0}^{n_i + \tilde{p}_{i}} \rho( f^j ( x_i )  )
     =  \frac{n_i}{n_i +  \tilde{p}_{i} }  \; \frac{1}{n_i}  \; \sum_{j=0}^{n_i-2} \rho( f^j ( x_i )  )
     + \frac{n_i}{n_i + \tilde{p}_{i}} \; \frac{1}{n_i}  \; \sum_{j=n_i-1}^{n_i+\tilde{p}_{i}+1} \rho( f^j ( x_i )  )
$$
Since the roof function $\rho$ is almost everywhere finite %assumed bounded from below
then $n_i=n_i((x_i,s_i),t_i) \to \infty$ as $t \to \infty$ and by Birkhoff's ergodic theorem it follows that for $\overline \mu$-almost every $(x_i,s_i)$ the limit in second term in the right hand side is zero. Together with ~\eqref{eq:distortnun}, this proves that $\lim_{\xi\to 0} \lim_{t_i \to \infty} \frac{1}{n_i}  \; T((x_i,s_i),t_i,\xi) =0$.

Let $x\in M$ be given by the gluing orbit property as above and let $s=s_1$.
We claim that
$
d_\rho( X_{t}  ({x}, s) , X_t(x_1,s_1) ) <\vep
$
for every  $t \in [0, t_i]$ and there exists $0\le p_i \le T(\vep)$ so that
$
d_\rho( X_{t+t_1+p_1}  ({x}, s_1) , X_t(x_2,s_2) ) <\vep
$
for every  $t \in [0, t_{i+1}]$.

Since the proof of the shadowing process is completely analogous to the proof
of the Claim in Subsection~\ref{subsecClaim} we will focus on the main difference which consists
of expliciting the choice of the gluing times $p_i>0$.
%just consider the case where $i=1$
and keep the notation of that subsection.
%\begin{proof}[Proof of the claim]
%The proof is completely analogous to the proof . As in this case, is enough to $i = 1$. We will use the notation of the Claim~\ref{subsecClaim}.
Thus we are reduced to prove the existence of $0\le p_i \le T((x_i,s_i),t_i,\xi)$ for which
$
d_\rho( X_{p_i}  (X_{t_i}({x}, s_i)) , (x_{i+1},s_{i+1}) ) <\vep
$.
As in the proof of the previous theorem, we subdivide the argument for each choice of the gluing time
in three cases.

\noindent
If $j=j_1$ take
    $$
    p_1
      =
    \begin{cases}
    \begin{array}{ll}
s_2 + \sum_{i=0}^{\tilde{p}_{1}-1} \rho(f^{i+j_1}(x))
    - [s_1 + t_1 - \sum_{i=0}^{ j_1-1} \rho(f^i(x))],
    & \text{ if } s_{2} \leq \rho(f^{\tilde{p}_{1}+ j_{1}}(x))
    \\
s_2 - C_{\xi} + \sum_{i=0}^{\tilde{p}_{1}-1} \rho(f^{i+ j_1}(x))
    + [ \sum_{i=0}^{j_1-1} \rho(f^i(x)) - (s_1+t_1)]
        & \text{ otherwise. }
    \end{array}
   \end{cases}
    $$
If $j=j_1-1$ take
    $$
    p_1
      =
    \begin{cases}
    \begin{array}{ll}
s_2 + \sum_{i=0}^{\tilde{p}_{1}-1} \rho(f^{i+j_{1} +1}(x))
    - [s_1 + t_1 - \sum_{i=0}^{ j_1} \rho(f^i(x))],
    & \text{ if } s_{2} \leq \rho(f^{\tilde{p}_{1}+ j_{1}+1}(x))
    \\
 s_2 - C_{\xi} + \sum_{i=0}^{\tilde{p}_{1}-1} \rho(f^{i+ j_{1}+1}(x))
    + [ \sum_{i=0}^{j_1} \rho(f^i(x)) - (s_1+t_1)]
        & \text{ otherwise. }
    \end{array}
   \end{cases}
    $$
If $j=j_1+1$ then take
    $$
    p_1
      =
    \begin{cases}
    \begin{array}{ll}
s_2 + \sum_{i=0}^{\tilde{p}_{1}-1} \rho(f^{i+j_{1}-1 }(x))
    - [s_1 + t_1 - \sum_{i=0}^{ j_1-2} \rho(f^i(x))],
    & \text{ if } s_{2} \leq \rho(f^{\tilde{p}_{1}+ j_{1}-1}(x))
    \\
s_2 - C_{\xi} + \sum_{i=0}^{\tilde{p}_{1}-1} \rho(f^{i+ j_{1}-1}(x))
    + [ \sum_{i=0}^{j_1-2} \rho(f^i(x)) - (s_1+t_1)]
        & \text{ otherwise. }
    \end{array}
   \end{cases}
    $$
    %\begin{itemize}
 %\item[(a)] if $j=j_1$ then take
   % \begin{align*}
  %  p_1
   % = s_2 + \sum_{i=0}^{\tilde{p}_{1}-1} \rho(f^{i+j_1}(x))
    %- [s_1 + t_1 - \sum_{i=0}^{ j_1-1} \rho(f^i(x))],
    %\end{align*}
    %if $s_{2} \leq \rho(f^{\tilde{p}_{1}+ j_{1}}(x))$, otherwise take
   %\begin{align*}
   % p_1
   % = s_2 - C_{\xi} + \sum_{i=0}^{\tilde{p}_{1}-1} \rho(f^{i+ j_1}(x))
   % + [ \sum_{i=0}^{j_1-1} \rho(f^i(x)) - (s_1+t_1)] \, \text{and}
    %\end{align*}

%\item[(b)] if $j=j_1-1$ then take
  %  \begin{align*}
   % p_1
   % = s_2 + \sum_{i=0}^{\tilde{p}_{1}-1} \rho(f^{i+j_{1} +1}(x))
    %- [s_1 + t_1 - \sum_{i=0}^{ j_1} \rho(f^i(x))],
   % \end{align*}
    %if $s_{2} \leq \rho(f^{\tilde{p}_{1}+ j_{1}+1}(x))$, otherwise take
   %\begin{align*}
   % p_1
   % = s_2 - C_{\xi} + \sum_{i=0}^{\tilde{p}_{1}-1} \rho(f^{i+ j_{1}+1}(x))
   % + [ \sum_{i=0}^{j_1} \rho(f^i(x)) - (s_1+t_1)] \, \text{and}
   % \end{align*}

%\item[(c)] if $j=j_1+1$ then take
  %  \begin{align*}
   % p_1
   % = s_2 + \sum_{i=0}^{\tilde{p}_{1}-1} \rho(f^{i+j_{1}-1 }(x))
    %- [s_1 + t_1 - \sum_{i=0}^{ j_1-2} \rho(f^i(x))],
    %\end{align*}
    %if $s_{2} \leq \rho(f^{\tilde{p}_{1}+ j_{1}-1}(x))$, otherwise take
  % \begin{align*}
   % p_1
   % = s_2 - C_{\xi} + \sum_{i=0}^{\tilde{p}_{1}-1} \rho(f^{i+ j_{1}-1}(x))
   % + [ \sum_{i=0}^{j_1-2} \rho(f^i(x)) - (s_1+t_1)] \, \text{and}
   % \end{align*}
%\end{itemize}
In all cases,
\begin{align*}
0 \le p_1 & \le \sum_{j=0}^{\tilde{p}_{1}+2} \rho(f^{j+ n_1-1}(x_1))
     \le \sum_{j=0}^{\tilde{p}_{1}+2} \rho(f^{j+ n_1-1}(x_1)) + C_\xi
    = T((x_1,s_1),t_1,\xi).
\end{align*}
This proves our claim and
%\end{proof}
 completes the
proof of the theorem.

%%%%%%%%%%%%%%%%%%%%%%%%%%%%
\subsection{Proof of Theorem~\ref{thm:nugluing2}}

Let $f: M \to M$ satisfy the non-uniform specification property with respect
to the ergodic and hyperbolic measure $\mu$ (c.f. \cite{OT,Va12}) and assume the roof
function $\rho: M \to \mathbb R_0^+$ is $\mu$-integrable and satisfies the
bounded distortion condition \eqref{eq:distortnun}.

By the non-uniform specification property for $(f,\mu)$, for $\mu$-almost every
$x$, every $\xi>0$ and $n\ge 1$ there exists $p(x,n,\xi) \ge 1$ satisfying $\lim_{\xi \to 0} \limsup_{n \to\infty} p(x,n,\xi)/n=0$
and such that the following property holds: for every $k\ge 1$,  $\mu^k$-almost every
$(x_1, \dots, x_k) \in M^k$, every $n_1, \dots, n_k \ge 1$ and $p_i \ge p(x_i,n_i,\xi)$
there exists $x\in M$ such that
$$
\begin{array}{cc}
d(f^j(x),f^j(x_1) ) \leq \xi
    \quad \text{and} \quad
    d (f^{j+n_1+p_1+\dots +n_{i-1}+p_{i-1}}(x) \;,\; f^j(x_i) ) \leq \xi
\end{array}
$$
for every $2\leq i\leq k$ and $0\leq j\leq n_i$.

We proceed to prove that the semiflow satisfies the non-uniform gluing orbit property. Fix $\vep>0$.
Consider arbitrary points $(x_1,s_1), (x_2,s_2), \dots, (x_k,s_k) \in M_r$ in a
full $\overline\mu^k$-measure set and times $t_1, \dots, t_k \ge 0$, in such a way that
$\lim_{\xi \to 0 C_\xi(x_i)=0}$ for every $i$. Let $\xi>0$ be such that $\xi+C_\xi(x_i) <\vep$ for every $1\le i \le k$.

Associated to $(x_i, s_i)$ and $t_i$ consider the dynamical balls $B(x_i, n_i,\xi) \subset M$,
%$B(x_i, n_i,\vep) \subset M$,
where each lap number $n_i=n_i(x_i, s_i,t_i)\ge 1$ is determined by
equation \eqref{eq:turns}. Let us define %the gluing time
$$
T((x_i,s_i),t_i,\xi)
    = \sum_{j=0}^{p(x_i,n_i,\xi)+1} \rho( f^{j+n_i-1}(x_i)) + C_\xi(x_i).
$$
We claim that $T((x_i,s_i),t_i,\xi)$ has sublinear growth in $t_i$.  Similarly to before one can write
\begin{align}
\frac{1}{n_i + p(x_i,n_i,\vep)+2}
     \sum_{j=0}^{n_i + p(x_i,n_i,\vep)} \rho( f^j ( x_i )  )
    & =  \frac{n_i-2}{n_i +  p(x_i,n_i,\vep) +2}  \; \frac{1}{n_i-2}  \; \sum_{j=0}^{n_i-2} \rho( f^j ( x_i )  )  \label{eq:4} \\
    & + \frac{1}{n_i + p(x_i,n_i,\vep)+2} \;  \sum_{j=0}^{p(x_i,n_i,\xi)+1} \rho( f^{j+n_i-1} ( x_i )  ) \label{eq:5}.
\end{align}
Since the roof function $\rho$ is bounded away from zero  then $n_i=n_i((x_i,s_i),t_i) \to \infty$ as $t_i \to \infty$ and
$$
\lim_{\vep\to 0} \lim_{t_i \to\infty} \frac{n_i}{n_i + p(x_i,n_i,\vep)}
    = \lim_{\vep\to 0}  \lim_{t_i \to\infty} \frac{1}{1+ \frac{p(x_i,n_i,\vep)}{n_i}} =1.
$$
Hence, by Birkhoff's ergodic theorem
the term \eqref{eq:5} tends to zero as $t_i \to \infty$ for $\mu$-almost every $x_i$.
Using $\lim_{\xi \to 0} C_\xi(x_i) =0$ together with the previous equality, it follows that
$\lim_{t_i \to \infty} \frac{1}{n_i}  \; T((x_i,s_i),t_i,\vep)  ) =0$.
We observe also that
%Using that
%the roof function $\rho$ is bounded away from zero and
$\inf \rho \, n_i \le \sum_{j=0}^{n_i} \rho(f^j(x)) \le s_i+t_i$ we deduce that
$n_i \le C \frac{t_i}{\inf\rho}$ and consequently
$
\lim_{\xi\to0} \lim_{t_i \to \infty} \frac{1}{t_i}  \; T((x_i,s_i),t_i,\xi)  ) = 0.
$  Since the remaining of the
proof follows the same lines of Theorem~\ref{thm:nugluing1}
we shall omit the details.

%%%%%%%%%%%%%%%%%%%%%%%
\section*{Appendix: On the tempered variation condition}

In this Appendix we relate the tempered variation condition for observables on suspension semiflows
with the corresponding condition for reduced observables on the base dynamics.

\begin{proposition}\label{prop1}
Let $(X_{t})_{t\ge 0}$ be a suspension semiflow over a dynamical system  $f : M\to M$
with a roof function $\rho$ that is bounded away from zero and infinity and has tempered variation.
%and $\lim_{n\to \infty}\frac{r(f^{n}(x))}{n} =0$, uniformly in $x$.
If the observable $\psi : M_\rho \to \mathbb R$ is bounded and the reduced observable
$\overline{\psi} : M \to \mathbb R$ has tempered variation then $\psi$ has tempered variation.
\end{proposition}

\begin{proof}
Given $t, \delta>0$ and points $(x,t_{1}) \in M_\rho$ and $(y,t_{2}) \in B((x,t_{1}) , t , \delta)$, using that $\inf \rho >0$,
there exists $n \in \NN_0$ such that either (i) $S_{n}\rho (x) \leq t$ and $S_{n}\rho(y) \le t < S_{n+1}\rho(y) $
or (ii) $S_{n}\rho (y) \leq t$ and $S_{n}\rho(x) \le t < S_{n+1}\rho(x)$.
Assume that (i) holds since the other case is completely analogous. Then
%$S_{n+1}\rho(x) > t \geq S_{n}\rho(x)$ and $S_{n} \rho(y) \leq t$ or
%$S_{n+1}\rho(y) > t \geq S_{n}\rho(y)$ and $S_{n}\rho (x) \leq t$. In the first case,
\begin{align*}
\Big| \int_{0}^{t}\psi(X_{s}(x,t_{1})) - \psi(X_{s}(y,t_{2})) ds \Big|
	& \le \Big |  \int_{0}^{S_{n}\rho(x)}\psi(X_{s}(x) ds - \int_{0}^{S_{n}\rho(y)}\psi(X_{s}(y) ds \,\Big| \\
	& + \Big |  \int_{S_{n}\rho(x)}^{t}\psi(X_{s}(x) ds   - \int_{S_{n}\rho(y)}^{t}\psi(X_{s}(y)ds \Big| \\
	& + \Big |  \int_{0}^{t_1} \psi(X_{s}(x) ds - \int_{0}^{t_2} \psi(X_{s}(y) ds \,\Big| \\
	%
	%& = \Big |  \int_{0}^{S_{n}\rho(x)}\psi(X_{s}(x)ds + \int_{S_{n}\rho(x)}^{t}\psi(X_{s}(x)ds \\
	%& -\int_{0}^{S_{n}\rho(y)}\psi(X_{s}(y)ds - \int_{S_{n}\rho(y)}^{t}\psi(X_{s}(y)ds \Big| + (t_{1} + t_{2})\sup|\psi| \\
	& \le |S_{n}\overline{\psi}(x) - S_{n}\overline{\psi}(y)| + \Big| \int_{S_{n}\rho(x)}^{S_{n}\rho(y)}\psi(X_{s}(x))ds \Big| \\
	& + \Big| \int_{S_{n}\rho(y)}^{t} \psi(X_{s}(x)) - \psi(X_{s}(y)) ds \Big| + (t_{1} + t_{2})\sup|\psi|.
\end{align*}
where we used $\bar\psi(x) = \int_{0}^{\rho(x)} \psi(X_{s}(x)) \, ds$.
Furthermore, we note that $y \in B_f(x,n,\delta)$ and that $\frac{t}{\inf\rho} \le n \le \frac{t}{\sup\rho}$ and, consequently, $n=n(t) \to \infty$ as $t\to\infty$.
This yields that
\begin{align*}
\Big| \frac1t \int_{0}^{t}\psi(X_{s}(x)) - \psi(X_{s}(y)) ds \Big|
	& \leq
	\frac{\sup |\psi|}{\inf \rho}
	  \Big[  \frac1n |S_{n}\overline{\psi}(x) - S_{n}\overline{\psi}(y)| +  \frac1n |S_{n}\rho(x) - S_{n}\rho(y)|  \Big] \\
	& +  \frac3n \frac{ \sup \rho \, \sup |\psi|}{\inf \rho}
	%2\frac{\rho(f^{n+1}(x))}{n} \sup |\psi| + (t_{1} + t_{2})\sup|\psi| \Big],
\end{align*}
tends to zero as $t\to\infty$. This proves that $\psi$ has tempered variation.
\end{proof}

\begin{proposition}
Let $M$ be a compact metric space, $(X_{t})_{t}$ be a suspension semiflow over a bi-Lipschitz homeomorphism $f: M\to M$
with a H\"older continuous roof function $\rho$ and let $\mu$ be an $f$-invariant probability measure.
Suppose that for every H\"older continuous observable $g$ there exists a constant $D> 0$ such that
$$
C_\epsilon(g)
	:= \sup_{n \in \NN}\sup_{y \in B(x , n , \vep)}\Big| \sum_{i=0}^{n-1}g(f^{i}(x)) - \sum_{i=0}^{n-1}g(f^{i}(y))\Big| \leq D\vep,
$$
for every $\vep > 0$, and that $\phi:M \rightarrow \R$ is a potential bounded. Then
\begin{itemize}
\item[(i)] If $\mu$ is a weak Gibbs measure for $f$ with respect to $\overline{\phi}$ then $(\mu \times \Leb) / \int \rho d\mu$ is weak Gibbs measure
  	for $(X_t)_{t\in\mathbb R}$ with respect to $\phi$.
\item[(ii)] If $\mu$ is an $f$-invariant Gibbs measure for $f$ with respect to $\overline{\phi}$ then
	$(\mu \times \Leb) / \int \rho d\mu$ is a Gibbs measure for $(X_t)_{t\in\mathbb R}$ with respect to $\phi$.
\end{itemize}
\end{proposition}

\begin{proof}
Under the assumptions of the proposition, is proven in \cite[Proposition 19]{BS} that there exists $\kappa>0$
so that for every $x\in M$, $0<s<\rho(x)$ and $m\in \mathbb N$ it holds that
$$
B_{M_\rho} \big((x,s), S_m \rho(x), \frac1\kappa \vep \big)
	\subset B_M(x,m, \vep) \times (s-\vep, s+\vep)
	\subset B_{M_\rho} \big((x,s), S_m \rho(x), \kappa \vep \big)
$$
for every sufficiently small $\vep$. In general, for $t>0$ there exists $m\in \mathbb N$, depending on $x$,
so that $S_m\rho(x) \le t+ s < S_{m+1} \rho(x)$ and so
$$
B_{M_\rho} \big((x,s), S_{m+1} \rho(x), \vep \big)
	\subset B_{M_\rho} \big((x,s), t,  \vep \big)
	\subset B_{M_\rho} \big((x,s), S_m \rho(x), \vep \big).
$$
This implies that
$$
\mu ( B_M(x,m+1, \frac\vep\kappa) \times (s-\frac\vep\kappa, s+\frac\vep\kappa)  )
	\le \mu (  B_{M_\rho} \big((x,s), t,  \vep \big)  )
	\le \mu ( B_M(x,m, \kappa\vep) \times (s-\kappa\vep, s+\kappa \vep)  )
$$
and the desired result follows by simple integration.
\end{proof}

%%%%%%%%%%%%%%%%%%%%%%%
\subsection*{Acknowledgments}
The authors are deeply grateful to V. Ara\'ujo and A. Arbieto for the incentive and suggestions that helped to improve the
structure of the paper, and to C. Zhao for pointing out a mistake in a previous version of Theorem~B.
The second author was partially supported by a CNPq-Brazil posdoctoral fellowship at University of Porto and
by CMUP (UID/MAT/00144/2013), which is funded by FCT (Portugal) with national (MEC) and European structural
funds through the programs FEDER, under the partnership agreement PT2020.

%%%%%%%%%%%%%%%%%%%%%%%%

\def\cprime{$'$}

\end{document}